\documentclass[12pt,a4paper]{article}

\usepackage[margin=2.25cm]{geometry}
\usepackage{graphicx}  
\usepackage{amsmath}  
\usepackage{amssymb,amsthm}
\usepackage[fixamsmath,disallowspaces]{mathtools}
\usepackage{color}
\usepackage{tikz}
\usepackage{pgfplots}
\pgfplotsset{compat=1.16}
\usepackage[pdfpagemode=UseOutlines]{hyperref}
\hypersetup{
    colorlinks=false,
    pdfborder={0 0 0},
    pdfauthor={T. Haas, B. de Rijk, G. Schneider},
   pdftitle={Validity of Whitham's modulation equations for dissipative systems with a conservation law: Phase dynamics in a generalized Ginzburg-Landau system}
}

\newcommand{\R}{{\mathbb R}}\newcommand{\N}{{\mathbb N}}

\newcommand{\C}{{\mathbb C}}

\newcommand{\F}{\mathcal{F}}

\def\F{\mathcal{F}}

\def\E{\mathcal{E}}

\def\Im{\mathrm{Im}}

\def\Chi{G}
\def\bu{{\bf u}}
\def\br{{\bf r}}
\def\bv{{\bf v}}

\def\({\left(}
\def\){\right)}
\def\eps{{\varepsilon}}
\def\veps{{\varepsilon}}
\def\Dt{\partial_T}

\let\epsilon\varepsilon
\let\theta\vartheta
\let\hat\widehat

\newtheorem{theorem}{Theorem}[section]\newtheorem{lemma}[theorem]{Lemma}

\newtheorem{corollary}[theorem]{Corollary}

\newtheorem{remark}[theorem]{Remark}

\title{
Validity of Whitham's modulation equations  \\ for dissipative systems with a conservation law\\
{\large -- Phase dynamics in a generalized Ginzburg-Landau system --}}
\author{Tobias Haas\footnotemark[2] \footnotemark[4]
\and Bj\"orn de Rijk\footnotemark[2] \footnotemark[3]
\and Guido Schneider\footnotemark[2] \footnotemark[5]}
\date{\vspace{-2em}}

\begin{document}

\maketitle

\renewcommand{\thefootnote}{\fnsymbol{footnote}}
\footnotetext[2]{Institut f\"ur Analysis, Dynamik und Modellierung, Universit\"at Stuttgart, Pfaffenwaldring 57, 70569 Stuttgart, Germany}
\footnotetext[4]{tobias.haas@mathematik.uni-stuttgart.de}
\footnotetext[3]{bjoern.derijk@mathematik.uni-stuttgart.de}
\footnotetext[5]{guido.schneider@mathematik.uni-stuttgart.de}

\begin{abstract}
It is well-established that Whitham's modulation equations approximate the dynamics of slowly varying periodic wave trains in dispersive systems.
We are interested in its validity in dissipative systems with a conservation law.
The prototype example for such a system is the generalized Ginzburg-Landau system that arises as a universal amplitude system for the description of a Turing-Hopf bifurcation in spatially extended pattern-forming systems with neutrally stable long modes.
In this paper we prove rigorous error estimates between the approximation obtained through Whitham's modulation equations and true solutions to this Ginzburg-Landau system.
Our proof relies on analytic smoothing, Cauchy-Kovalevskaya theory,
energy estimates in Gevrey spaces, and a local decomposition in Fourier space, which separates center from stable modes and uncovers a (semi)derivative in front of the
relevant nonlinear terms.


\end{abstract}

%
%
%
%


\section{Introduction}

Modulation, envelope, or amplitude equations, such as the Ginzburg-Landau equation, the KdV equation, and the NLS equation, play an important role in the qualitative and quantitative description of spatially extended dissipative or conservative physical systems.
Mathematical theorems show that these asymptotic models make correct predictions about the dynamics of the original systems on sufficiently long time scales.
Examples of regimes which can be described in such a way are pattern-forming systems close to their first instability,
the long-wave limit of the water wave problem, or highly oscillatory regimes in nonlinear optics, see~\cite{CE90,vH91,Cr85,Kal87}, and~\cite{SU17book} for a recent overview.

Asymptotic models can also be employed to approximate the dynamics of slow modulations of periodic wave trains.
If the underlying system is dissipative, approximation theorems have for instance been shown in~\cite{MS04a,DSSS09}, where
the phase diffusion equation and Burgers equation have been justified as asymptotic models.

In this paper we focus on Whitham's modulation equations (WMEs), cf.~\cite{Whithambuch}, as asymptotic models for the description of slowly varying periodic wave trains.
Although the WMEs were initially derived by Whitham for conservative systems~\cite{Wh65a,Wh65b}, they arise as asymptotic models in larger classes of problems.
As of today, the application and derivation of WMEs, so-called Whitham modulation theory, is still subject of active research~\cite{Br17}.
The simplest possible conservative example where WMEs can be derived is the NLS equation or a system of coupled NLS equations, see~\cite{DS09,BKS20,BKZ20} for rigorous approximation results.

If dissipative systems are coupled with a conversation law, then WMEs can be derived as asymptotic models for the description of slow modulations of periodic wave trains, too.
It can be shown, e.g.~by rigorous Evans function computations~\cite{JZ10} or by direct Bloch wave expansions~\cite{JZB10}, that spectral stability of periodic wave trains against large-wavelength perturbations is equivalent to local well-posedness of the associated asymptotic WMEs in Sobolev spaces.
Thus, Whitham's modulation theory has been employed to establish spectral stability of periodic wave trains in various settings~\cite{Se05,JNRZ15,JNRYZ19}.
In addition, WMEs have been used in~\cite{JZN11,JNRZ13,JNRZ14} to investigate the nonlinear stability of time- and space-periodic solutions to pattern-forming systems with a conservation law against localized and nonlocalized perturbations.

This begs the question whether a rigorous approximation result for the WMEs can be obtained in such dissipative systems coupled with a conservation law. To the authors' best knowledge, such an approximation result has not been established so far.

In this paper we make a first step in this direction by establishing such an approximation result for the generalized Ginzburg-Landau system that arises as a universal amplitude system for pattern-forming systems with a conservation law close to a first instability of Turing-Hopf type, see~Section~\ref{sec11}.
This generalized Ginzburg-Landau (gGL) system consists of the classical complex Ginzburg-Landau equation coupled with a conservation law. In normalized form it reads
\begin{align}
\begin{split}
 \partial_t A &= (1+ i \alpha  )\partial_x^2 A + A -(1+i \beta ) A |A|^2 + ( \gamma_r + i \gamma_i  ) A B, \\
 \partial_t B &= a \partial_x^2 B + c \partial_x B + d \partial_x (|A|^2),
 \end{split}\label{eq1}
\end{align}
with coefficients $ \alpha,\beta,\gamma_r,\gamma_i,c,d \in \mathbb{R} $, $ a > 0 $, and
with $ A(x,t) \in \mathbb{C} $ and $ B(x,t) \in \mathbb{R} $. The gGL system possesses a two-parameter family of wave-train solutions, which are given by
\begin{equation}
(A,B) = \left(e^{\rho-i\omega t+iqx},b\right), \label{phoe0}
\end{equation}
where $b, \omega,q,\rho \in \R$ satisfy
\begin{equation}\label{phoe1}
- \omega = - \alpha q^2 - \beta e^{2 \rho} + \gamma_i b
  \qquad \textrm{and} \qquad
0 = -q^2 +1 - e^{2 \rho} + \gamma_r b.
\end{equation}


\subsection{Pattern forming systems  with a conservation law}

\label{sec11}
The B\'enard-Marangoni problem~\cite{Ta81}, the flow down an inclined plane~\cite{CD02}, or the Faraday experiment~\cite{ANR14} are
examples of pattern-forming systems with a conservation law.
The WMEs can formally be derived to describe slow modulations in time and space of periodic wave trains
appearing in such systems.

Due to the conservation law such systems possess a family of ground states whose linearizations
have, even in the stable case, continuous spectrum up to the imaginary axis for all values of the bifurcation parameter, see Figure~\ref{fig1}.
The wave trains, in which we are interested, are generated when the ground states destabilize through a short wave instability leading to a Turing-Hopf bifurcation.
For parameter values close to this instability the gGL system~\eqref{eq1} can be derived via a multiple scaling analysis as a universal amplitude system describing slow modulations in time and space of the most unstable linear modes.
In~\cite{HSZ11} it has been shown for a toy problem that the gGL system~\eqref{eq1}
makes correct predictions about the dynamics of pattern-forming systems with a conservation law close to the first instability.

Thus, motivated by its descriptive properties, we will focus on the gGL system~\eqref{eq1} as original system in this paper to study slow modulations in time and space of its periodic wave-train solutions.
We do emphasize that the gauge-symmetry of~\eqref{eq1} simplifies the subsequent analysis because it allows to extract the
local wave number of the solutions in a trivial way. For general pattern-forming systems with a conservation law this step can be rather
technical, cf.~\cite{DSSS09}.

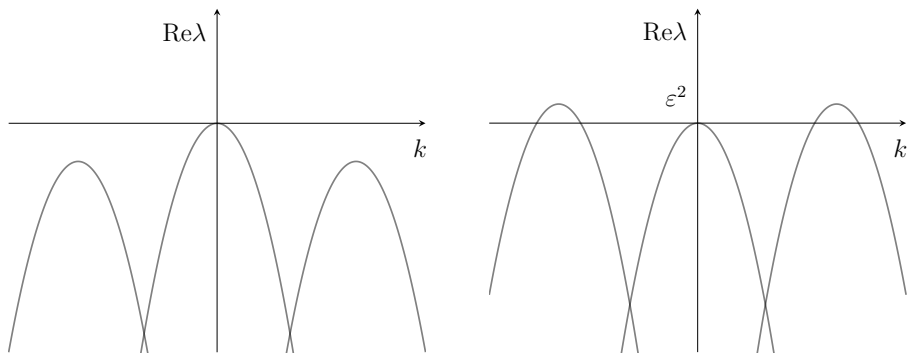
\begin{figure}[h] 
   \centering
     \begin{tikzpicture}[rotate=0,scale=0.8]
    \begin{axis}[
	xmin=-1.5, xmax=1.5,
	ymin=-0.6, ymax=0.3,
	axis lines=center,
	ticks=none,
    ]
      \addplot+[black,thick, domain=-0.6:0.6,samples=100,no marks,opacity=0.5] {-2*x^2};
         \addplot+[black,thick, domain=0.4:1.5,samples=100,no marks,opacity=0.5] {-2*(x-1)^2-0.1};
           \addplot+[black,thick, domain=-1.5:-0.4,samples=100,no marks,opacity=0.5] {-2*(x+1)^2-0.1};
           \node[below,black] at (axis cs:1.46,-0.02) {$ k $};
      \node[] at (axis cs:-0.35,0.45) {$\textrm{Re}(\lambda_+)$};
      \node[left=4pt] at (axis cs:0.05,0.24) {$ \textrm{Re} \lambda $};
    \end{axis}
  \end{tikzpicture}\qquad
    \begin{tikzpicture}[rotate=0,scale=0.8]
    \begin{axis}[
	xmin=-1.5, xmax=1.5,
	ymin=-0.6, ymax=0.3,
	axis lines=center,
	ticks=none,
    ]
      \addplot+[black,thick, domain=-0.6:0.6,samples=100,no marks,opacity=0.5] {-2*x^2};
         \addplot+[black,thick, domain=0.4:1.5,samples=100,no marks,opacity=0.5] {-2*(x-1)^2+0.05};
           \addplot+[black,thick, domain=-1.5:-0.4,samples=100,no marks,opacity=0.5] {-2*(x+1)^2+0.05};
           \node[below,black] at (axis cs:1.46,-0.02) {$ k $};
      \node[left=4pt] at (axis cs:0.05,0.24) {$ \textrm{Re} \lambda $};
          \node[left=4pt] at (axis cs:0.05,0.07) {$ \varepsilon^2 $};
    \end{axis}
  \end{tikzpicture}

    \caption{Solutions of the linearization about the trivial ground state of translationally invariant pattern-forming systems
   are proportional to $e^{ikx+\lambda_j(k) t} $.
   The left panel shows the relevant spectral curves $ k \mapsto \lambda_j(k) $ in the stable situation for systems with a conservation law, whereas the right panel depicts the unstable situation.}
   \label{fig1}
\end{figure}

\begin{remark}{\rm
To get some intuition about the role of $ A $ and $ B $ in~\eqref{eq1}
let us consider the flow down an inclined plane.
For this problem the ground states are given by the family of Nusselt solutions, which have a parabolic flow profile, with different constant fluid heights.
The complex amplitude $ A $ in~\eqref{eq1} describes slow modulations in time and space of the underlying bifurcating spatially periodic pattern.
The amplitude $ B $ describes slow modulations in time and space of the underlying background state. For the inclined plane problem these are variations of the averaged mean height of the fluid. }
\end{remark}

\subsection{Slow modulations of periodic wave trains}

The gGL system~\eqref{eq1} possesses a family of wave-train solutions given by~\eqref{phoe0}-\eqref{phoe1}.
Wave trains propagate through the spatial domain with constant speed while maintaining their periodic profile, and thus are periodic in time and space.
One readily verifies that such wave-train solutions only exist in~\eqref{eq1} when the trivial solution $ (A,B) = (0,b) $ is unstable.

W.l.o.g.~we consider $ b = 0 $ in the following.
%
%
Indeed, by setting $ (A,B) = \smash{(  \widetilde{A}, b + \widetilde{B})} $
we transform the equilibrium  $ (A,B) = (0,b) $ of the $ (A,B) $-system~\eqref{eq1} into the equilibrium
$\smash{ ( \widetilde{A},\widetilde{B}) = (0,0) }$ of the $\smash{ (\widetilde{A},\widetilde{B})}$-system.
By this transformation
the additional  term $\smash{ ( \gamma_r + i \gamma_i  ) b  \widetilde{A}  }$ appears in the $\smash{  \widetilde{A} }$-equation.
By setting $\smash{ \widetilde{A} = \breve{A} e^{i \gamma_i   b t} }$ the part  $\smash{  i \gamma_i   b  \widetilde{A}   }$ can be removed
from the $ \smash{ \widetilde{A} }$-equation. In addition, using that in the relevant regime~\eqref{phoe1} it holds $1+\gamma_r b = q^2 + e^{2\rho} > 0$, the $ \smash{(\breve{A} ,  \widetilde{B}) }$-system
can be brought back to the normal form~\eqref{eq1} with transformed
coefficients $ \gamma $, $ a $, $ c $, and $ d $ upon rescaling $\smash{ \breve{A} , \widetilde{B} }$, $ t $ and $ x $.

We are interested in the dynamics near the
family of wave trains given by~\eqref{phoe0}-\eqref{phoe1}. In particular,
we are interested in slow long-wave modulations of
these solutions.
For our purposes  it turns out to be advantageous
to work in
polar coordinates of the form $ A = e^{\widetilde{r}+ i \widetilde{\varphi}} $, in which the gGL system~\eqref{eq1} reads
\begin{align*}
\partial_t \widetilde{r} & = \partial_x^2 \widetilde{r} - \alpha \partial_x^2 \widetilde{\varphi} + 1- e^{2\widetilde{r}} + (\partial_x \widetilde{r})^2- (\partial_x \widetilde{\varphi})^2
-2 \alpha (\partial_x \widetilde{\varphi})(\partial_x \widetilde{r}) + \gamma_r B, \\
\partial_t \widetilde{\varphi} & = \partial_x^2 \widetilde{\varphi} + \alpha \partial_x^2 \widetilde{r} - \beta  e^{2\widetilde{r}} + \alpha  (\partial_x \widetilde{r})^2- \alpha  (\partial_x \widetilde{\varphi})^2 + (\partial_x \widetilde{\varphi})(\partial_x \widetilde{r}) + \gamma_i B, \\
\partial_t B  & = a\partial_x^2 B + c \partial_x B + d \partial_x (e^{2\widetilde{r}}).
\end{align*}
In  these polar coordinates
the family of wave-train solutions~\eqref{phoe0}
is given by $ \widetilde{r} = \rho $, $ \widetilde{\varphi} = -\omega t +qx$, and $ B = 0 $.
The deviation from these solutions is then introduced by $ \widetilde{\varphi} =  - \omega t +qx + \phi$,  $\tilde r = \rho + r$ and $\widetilde B = 0 + B$.
Thus, using~\eqref{phoe1} to simplify, the modulation $(r,\phi,B)$ of the wave trains satisfies the equations
\begin{align*}
\partial_t r & = \partial_x^2 r - \alpha \partial_x^2 \phi + e^{2\rho}(1- e^{2r}) + (\partial_x r)^2
- 2 q (\partial_x \phi) - (\partial_x \phi)^2\\&  \qquad
-2 \alpha q(\partial_x r) -2 \alpha (\partial_x \phi)(\partial_x r) + \gamma_r B, \\
\partial_t \phi & = \partial_x^2 \phi + \alpha \partial_x^2 r + \beta e^{2\rho}(1 - e^{2r}) + \alpha  (\partial_x r)^2- 2\alpha q (\partial_x \phi)- \alpha  (\partial_x \phi)^2
\\ &  \qquad
+2 q(\partial_x r) +2 (\partial_x \phi)(\partial_x r) + \gamma_i B, \\
\partial_t B  & = a \partial_x^2 B + c \partial_x B + d e^{2\rho} \partial_x (e^{2r}).
\end{align*}

\subsection{Formal derivation of the WMEs}

\label{sec13}

The WMEs, which we derive in the following, describe the evolution of the
local wave number $ \psi = \partial_x \phi $, for which we find
 \begin{align} \label{hamb1}
  \begin{split}
\partial_t r & =  \partial_x^2 r - \alpha \partial_x \psi + e^{2\rho}(1- e^{2r}) + (\partial_x r)^2
-2 q  \psi -  \psi^2
\\& \qquad  - 2 \alpha q (\partial_x r)-2 \alpha  \psi(\partial_x r) + \gamma_r B , \\
\partial_t \psi & = \partial_x^2 \psi + \alpha \partial_x^3 r  + \beta  e^{2\rho}\partial_x (1- e^{2r}) + \alpha \partial_x( (\partial_x r)^2)\\& \qquad- 2 \alpha  q \partial_x \psi- \alpha  \partial_x( \psi^2)  +2 q \partial_x^2 r
 +2 \partial_x(\psi(\partial_x r)) + \gamma_i \partial_x B,  \\
\partial_t B  & =  a\partial_x^2 B + c \partial_x B + d e^{2\rho} \partial_x (e^{2r}).
\end{split}
\end{align}
With $ T = \veps t $ and $ X = \veps x $ we make the long-wave approximation
\begin{equation*} 
r = \check{r}(X,T), \qquad \psi =  \check{\psi}(X,T), \qquad \textrm{and} \qquad B = \check{B}(X,T),
\end{equation*}
where $0 < \veps \ll 1$ is a small parameter. Under this scaling~\eqref{hamb1} transforms into
 \begin{align} \label{hann1}
 \begin{split}
\veps \partial_T \check{r} & =  \veps^2 \partial_X^2 \check{r} - \veps \alpha \partial_X \check{\psi} + e^{2\rho}(1- e^{2\check{r}}) + \veps^2 (\partial_X \check{r})^2
-2 q  \check{\psi} -  \check{\psi}^2
\\& \qquad  - 2 \veps \alpha q (\partial_X \check{r})-2 \alpha \veps  \check{\psi}(\partial_X \check{r}) + \gamma_r \check{B}, \\
\partial_T \check{\psi} & =  \veps \partial_X^2 \check{\psi} +  \veps^2 \alpha \partial_X^3 \check{r}  + \beta  e^{2\rho}\partial_X (1- e^{2\check{r}}) +  \veps^2 \alpha \partial_X( (\partial_X \check{r})^2)\\& \qquad- 2 \alpha  q \partial_X \check{\psi}- \alpha  \partial_X( \check{\psi}^2)  +2  \veps q \partial_X^2 \check{r}
 +2  \veps \partial_X(\check{\psi}(\partial_X \check{r})) + \gamma_i \partial_X \check{B},  \\
\partial_T \check{B}  & =   \veps a\partial_X^2 \check{B} + c \partial_X \check{B} + d e^{2\rho} \partial_X (e^{2\check{r}}).
\end{split}
\end{align}
At $ \mathcal{O}(\veps^0) $ we find
 \begin{align} \label{wyk2}
 0 & =  -2 q \check{\psi} - (\check{\psi})^2 + e^{2\rho}(1-  e^{2 \check{r}})+ \gamma_r   \check{B}, \\
 \partial_{T} \check{\psi} & =  \beta \partial_{X}( e^{2\rho} (1- e^{2 \check{r} }))-2 \alpha  q \partial_{X} \check{\psi}- \alpha  \partial_{X}( \check{\psi}^2)
 + \gamma_i  \partial_{X} \check{B}, \label{wyk3}
 \\
\partial_{T}  \check{B} & =   c \partial_{X} \check{B} + d e^{2\rho} \partial_{X} (e^{2\check{r}}). \label{wyk4}
 \end{align}
Equation~\eqref{wyk2} can be solved explicitly w.r.t.~$ e^{2 \check{r}} $.
Inserting this into the
equations~\eqref{wyk3} and~\eqref{wyk4} for $ \check{\psi} $ and $ \check{B} $ finally gives the WMEs,
a system of first order conservation laws,
\begin{align} \label{whi1b}
\begin{split}
 \partial_{T} \check{\psi} & =  \partial_{X}(  \beta  ( 2 q\check{\psi} +(\check{\psi})^2-\gamma_r   \check{B} ) - 2 \alpha  q \check{\psi} - \alpha  \check{\psi}^2
 + \gamma_i  \check{B}), \\
\partial_{T}  \check{B} & =   \partial_{X} (c \check{B} + d  (
- 2q\check{\psi}- (\check{\psi})^2+ \gamma_r   \check{B} )  ).
\end{split}
 \end{align}
Depending on the
coefficients, (the linearization of)~\eqref{whi1b}
has a hyperbolic or an elliptic character, cf.~Remark~\ref{remhypel}.
Since~\eqref{whi1b} only contains first derivatives,
 solutions
can be constructed with the Cauchy-Kovalevskaya theorem for analytic initial conditions
independent of the  character of~\eqref{whi1b}.
It is the goal of this paper to prove that~\eqref{whi1b} makes correct predictions about the dynamics
of~\eqref{hamb1}. Such a proof is a non-trivial task, since solutions of order
$ \mathcal{O}(1)$ have to be controlled on an $ \mathcal{O}(1/\veps) $-time scale, cf.~Remark~\ref{remNT}.

\subsection{The functional analytic set-up} 

In order to  state our main result we have to introduce a number of function spaces and notations, where
in most cases  we do not distinguish between
scalar and vector-valued functions.

The Fourier transform of a function $u$ is denoted by
\begin{align*}
\F(u)(k) = \hat{u}(k) = \frac{1}{\sqrt{2\pi}} \int_\R e^{-ikx} u(x)dx. 
\end{align*}
%
%

Due to the possible elliptic character of~\eqref{whi1b}
we work in Gevrey spaces $\Chi_\sigma^m$, which are defined by the inner product
\begin{align*} 
( u,v)_{\Chi^m_\sigma} &= \int_\R e^{2\sigma (1+|k|)}\left(1+|k|^{2}\right)^{m} \widehat u(k) \overline{\widehat{v}(k)} dk,
\end{align*}
for $\sigma \geq  0$ and  $m  \geq 0$.
 In the subsequent proofs, we use that
$G_\sigma^m$ is an algebra for $m>1/2$ and if $u,v\in G_\sigma^m$ then

\begin{equation}\label{1.alg}
\|uv\|_{G_\sigma^m}\le C_m\|u\|_{G_\sigma^m}\|v\|_{G^m_\sigma},
\end{equation}
where the constant $C_m > 0$ is independent of $\sigma\geq0$. In the vector-valued
case the product is replaced by an inner product on $\R^d$.
The formula (\ref{1.alg}) can be improved to
\begin{equation} \label{1.kp}
\|uv\|_{G^{m_1}_\sigma}\le C_{m_1,m_2} \(\|u\|_{G^{m_1}_\sigma}\|v\|_{G^{m_2}_\sigma}+\|u\|_{G^{m_2}_\sigma}\|v\|_{G^{m_1}_\sigma}\),
\end{equation}
which holds for all $\sigma \ge0$ and $m_{1,2} > 1/2$, where the constant $C_{m_1,m_2}$ is independent of $\sigma \geq 0$.

Moreover, for any entire function $\phi$ with  $\phi(0)=0$ and any $m > 1/2$, there is an entire function $\phi_m(z)$, which is monotonically increasing on $\R_{+}$ and satisfies $\phi_m(0)=0$, such that
for all $  u\in G^m_{\sigma} $ we have
\begin{equation}\label{1.f}
\|\phi(u)\|_{G^m_\sigma}\le \phi_m(\|u\|_{G^m_\sigma}).
\end{equation}
Functions $ u \in G^m_\sigma$  can be extended to functions that are analytic on the strip $\{z \in \C \colon |\Im(z)| < \sigma\}$ by the Paley-Wiener Theorem, cf.~\cite[Theorem IX.13]{RS75B2}.
It is easily seen that for any $\sigma_1  > \sigma_2 \geq 0 $ and any $m \geq 0$ we have the
 continuous embedding
$G^0_{\sigma_1} \hookrightarrow G^m_{\sigma_2} $.


\subsection{Main results and some remarks} 

For the formulation of our main result, we introduce $ V = (r,\psi,B)$, and abbreviate~\eqref{hamb1} as
\begin{equation} \label{nea1}
\partial_t V= L V +N(V),
\end{equation}
where $ LV$ stands for the linear and $ N(V)$ for the nonlinear terms.

Now let $(\check{\psi},\check{B})(X,T)$ be a solution to the WMEs~\eqref{whi1b}, whose local existence and uniqueness in Gevrey spaces follows from Cauchy-Kovalevskaya theory, see~\S\ref{sec-CK-gevrey}.
Then, as outlined in~\S\ref{sec13}, upon defining the function $\check{r}(X,T)$ as the corresponding solution to the algebraic equation~\eqref{wyk2}, we obtain the formal WMEs approximation
\begin{align}
V_{app}^{*,\epsilon}(x,t) = \left(\check{r},\check{\psi},\check{B} \right)(\veps x , \veps t), \label{formalapprox}
\end{align}
for long-wave modulations of wave-train solutions to the gGL system.

The following approximation result shows that, if $(\check{\psi},\check{B})(X,T)$ is sufficiently small in the appropriate Gevrey norm, then the WMEs approximation $V_{app}^{*,\epsilon}(x,t)$ approximates a modulation $V(x,t)$ satisfying~\eqref{nea1} on the natural $\mathcal{O}(1/\veps)$-time scale. Thus, the WMEs make correct predictions about the dynamics of slow modulations of the wave-train solutions~\eqref{phoe0} to the gGL system~\eqref{eq1}, cf.~Remark~\ref{recoverphaserem}.

\begin{theorem} \label{mainresult1}
For  $m > 3/2$, $\sigma_{0}, T_0 > 0$ there exists a constant $ C_{wh} > 0 $ such that the following holds.
Let
$$(\check{\psi},\check{B}) \in C^1\big((0, T_0],\Chi^m_{\sigma_0}\big) \cap C\big([0,T_0],\Chi^{m+1}_{\sigma_0}\big),$$
be a solution to the WMEs~\eqref{whi1b}  with
\begin{align} \label{whismal}
\sup_{T \in [0,T_0]} \| (\check{\psi},\check{B})(\cdot,T) \|_{G^{m+1}_{\sigma_{0}}} \leq {C_{wh}},
\end{align}
and let $\check{r}(X,T)$ be the corresponding solution to the algebraic equation~\eqref{wyk2}. Then, there exist $C,T_1, \veps_0 > 0$ such that for all $\veps \in (0,\veps_0)$, there exists a solution
\begin{align}\label{regV} V \in C^1\big((0,T_1/\veps],H^m(\R)\big) \cap C\big([0,T_1/\veps],H^{m+2}(\R)\big),\end{align}
to~\eqref{nea1} with
\begin{align}
 \sup_{t \in [0,T_1/\veps]} \sup_{x \in \mathbb{R}} |
 V(x,t) - V_{app}^{*,\veps}( x,t) | \leq C \veps^{1/2},
\label{mr2}
\end{align}
where $V_{app}^{*,\epsilon}(x,t)$ is given by~\eqref{formalapprox}.
\end{theorem}
Our proof is based on analytic smoothing, Cauchy-Kovalevskaya theory,
energy estimates in Gevrey spaces, and a local decomposition in
Fourier space separating center from stable modes and uncovering a semiderivative in front of the relevant nonlinear terms. This semiderivative,  which corresponds to a Fourier multiplier vanishing at frequency $k = 0$, allows us to control the center modes on the long $\mathcal{O}(1/\veps)$-time scale.
We note that, in order to obtain a sufficiently small residual, higher-order terms have to be added. Thus, we construct, for every given fixed $\kappa > 1$, a higher-order approximation for which the error bound in~\eqref{mr2} is $\leq C \veps^{\kappa - 1/2} $ instead of $\leq C\veps^{1/2}$, cf.~\S\ref{sec-approx}.

\begin{remark}\label{remNT}
{\rm
We note that obtaining the above validity result is a nontrivial task. The WMEs approximation $V_{app}^{*,\epsilon}(x,t)$, and thus also the associated solution $V(x,t)$ to~\eqref{nea1}, are of order $ \mathcal{O}(1) $ for $ \veps \to 0 $. Therefore, a simple application of Gr\"onwall's inequality would only provide the boundedness of the solutions on an $ \mathcal{O}(1) $-time scale, but not on the natural $ \mathcal{O}(1/\veps) $-time scale of the WMEs approximation.

In this context, we emphasize that there is a number of counterexamples where formally derived amplitude equations make wrong predictions on the natural time scale, cf.~\cite{Schn96MN,SSZ14,HS20}. Although Theorem~\ref{mainresult1} is not optimal in the sense that the possible approximation time $ T_1/\veps $ is possibly smaller than  $ T_0/\veps $,
 we do establish an approximation result on the natural $ \mathcal{O}(1/\veps) $-time scale of the WMEs approximation $V_{app}^{*,\epsilon}(x,t)$.
  }
\end{remark}
 \begin{remark}{\rm
It is a natural question whether the Gevrey spaces can be replaced
by classical  Sobolev spaces,
i.e., what happens if we give up analyticity on a strip in the
complex plane and choose the solutions to the WMEs to be only finitely many
times differentiable.
We have no answer at this point and have to postpone the
question to future research.
However, such an improved result can only be true under the additional
assumption that the wave train is spectrally stable, cf.~\cite{BKZ20}.
 }
\end{remark}
\begin{remark} \label{recoverphaserem} {\rm
In order to formulate the approximation result in the original $ (A,B) $-variables
we have  to regain the phase $\phi$ from the local wave number $\psi = \partial_x \phi$.
Integrating the pointwise bound from Theorem~\ref{mainresult1} leads to an error bound on a spatial interval of length $\mathcal{O}(\epsilon^{-\rho})$ with  $\rho \in (0,\kappa-1/2)$  rather than an $\R$-uniform error bound.
Moreover, we have to allow for a global phase $ e^{i \phi(0,t)} $.
These two restrictions  have been observed in a number of papers before.
For a detailed discussion we refer to~\cite[Section 4]{MS04b} or~\cite[Corollary 2.2]{DS09}.
}
\end{remark}

The plan of the paper is as follows. Section~\ref{sec2bks} focusses on the
construction and properties of
approximate solutions.
In Section~\ref{secdrei} we estimate the difference between these approximate
solutions and exact solutions to the gGL system, and thereby prove Theorem~\ref{mainresult1}.
\medskip

{\bf Notation.} In the following many possibly different constants are denoted by the same
symbol $ C $ as long as they can be chosen independent of the small perturbation parameter
$ 0 < \veps \ll 1 $.
\medskip

{\bf Acknowledgement.} This work is supported by the Deutsche Forschungsgemeinschaft DFG through the Sonderforschungsbereich  1173 ,,Wave phenomena: analysis and numerics''.

\section{Construction of approximate solutions}

\label{sec2bks}

In this section we collect the necessary Cauchy-Kovalevskaya theory in Gevrey spaces to yield local existence and uniqueness of solutions to the WMEs~\eqref{whi1b}.
In addition, we construct higher-order approximations to the slow modulations, which are needed for the proof of our main result, Theorem~\ref{mainresult1}.
Local existence of these higher-order approximations in Gevrey spaces is established by another application of Cauchy-Kovalevskaya theory.

\subsection{The structure of the problem} 

We start by putting the WMEs in the right framework for Cauchy-Kovalevskaya theory following~\cite{BKS20}. Thus, we switch to $(X,T)$-coordinates in the original system~\eqref{hamb1}, and observe that~\eqref{hann1} can be written as
\begin{align} \label{hann4}
\begin{split}
0&= {\bf M}_r(\br,\bu) +\eps  {\bf F}_r(D^k_X\br,D^k_X\bu,\partial_T \br), \\
\Dt \bu&={\bf M}_u(\br,\bu)\partial_X (\br,\bu)+\eps {\bf F}_u(D^k_X\br,D^k_X\bu),
 \end{split}
\end{align}
  where $ \br= r $ and
  ${\bf u} = (\psi,B)$,  where  $ {\bf M}_r(\br,\bu) $ and  $ {\bf M}_u(\br,\bu) $  are  entire
  functions of their arguments,
where $ {\bf F}_u $ is polynomial in
  $D_X^k\br=(\br,\partial_X\br,\ldots,\partial_X^k\br)$ and
  $D_X^k\bu=(\bu,\partial_X\bu,\ldots,\partial_X^k\bu)$ with $k = 3$,
and where $ {\bf F}_r $ is polynomial in  $D_X^k\br$,
  $D_X^k\bu$, and linear in $ \partial_T \br $.
 In detail, we  have
 \begin{align*}
 {\bf M}_r(\br,\bu) & =  -2 q \check{\psi} - (\check{\psi})^2 + e^{2\rho}(1-  e^{2 \check{r}})+ \gamma_r   \check{B}, \\
  \eps  {\bf F}_r(D^k_X\br,D^k_X\bu,\partial_T \br)  & =
 - \veps \partial_T \check{r} + \veps^2 \partial_X^2 \check{r} - \veps \alpha \partial_X \check{\psi}  - 2 \veps \alpha q (\partial_X \check{r})-2 \alpha \veps  \check{\psi}(\partial_X \check{r})
  , \\
 {\bf M}_u(\br,\bu) & =
  \left(\begin{array}{ccc}
   - 2  \beta  e^{2\rho} e^{2\check{r}}
  &  - 2 \alpha  q   - 2 \alpha  \check{\psi}
&   \gamma_i ,
  \\
 2 d e^{2\rho} e^{2\check{r}} & 0 & c
  \end{array} \right)
 , \\
  \eps  {\bf F}_u(D^k_X\br,D^k_X\bu)  & =
  \left(\begin{array}{c}
   \veps \partial_X^2 \check{\psi} +  \veps^2 \alpha \partial_X^3 \check{r}+  \veps^2 \alpha \partial_X( (\partial_X \check{r})^2)+2  \veps q \partial_X^2 \check{r}  +2  \veps \partial_X(\check{\psi}(\partial_X \check{r}))
  \\
   \veps a\partial_X^2 \check{B}
   \end{array} \right).
\end{align*}
In our notation we suppressed the explicit dependence of the functions $  {\bf F}_r $ and $  {\bf F}_u $
on the small parameter $ \eps $.
The advantage of the form~\eqref{hann4} is that for $\eps=0$
  it reduces to~\eqref{wyk2}-\eqref{wyk4}, namely
  \begin{align} \label{hann6}
0&= {\bf M}_r(\br,\bu) , \\
\Dt \bu&={\bf M}_u(\br,\bu)\partial_X (\br,\bu) .
 \label{hann7}
\end{align}
We solve~\eqref{hann6} w.r.t.~$ \br = \br^*(\bu)
$ for $ \bu $
sufficiently small. Inserting this into~\eqref{hann7} yields the WMEs~\eqref{whi1b}, now written as
  \begin{equation*} 
    \partial_T \bu^* = {\bf M}(\bu^*)\partial_X\bu^*    = {\bf M}_u(\br^*(\bu^*),\bu^*)\partial_X (\br^*(\bu^*),\bu^*)\,,
  \end{equation*}
 with
 \begin{align} \label{above}
 {\bf M}(\bu^*) =
  \left(\begin{array}{cc}
  2 (\beta-\alpha) q + 2 (\beta-\alpha) \check{\psi}
  &  - \beta \gamma_r+ \gamma_i
  \\
  -2dq- 2d  \check{\psi}
    & c + d \gamma_r
    \end{array} \right)
 \end{align}
an entire function of its arguments.

\subsection{Cauchy-Kovalevskaya theory in Gevrey spaces}

\label{sec-CK-gevrey}

The following Cauchy-Kovalevskaya theorem provides local existence and uniqueness in Gevrey spaces for the quasilinear abstract Cauchy problem of the form
\begin{equation}\label{1.PDE}
\partial_T \bu={\bf M}(\bu)\partial_X \bu,\ \ \bu\big|_{T=0}=\bu_0\,, \qquad X \in \R, \, T \geq 0,
\end{equation}
where $\bu(X,T)$ is the unknown function taking values in $\R^d$, the initial condition $\bu_0$ lies in the Gevrey space $\Chi^{m}_{\sigma_0}$ for some $m>1$ and $\sigma_0 > 0$, and ${\bf M}(\bu)$ is an entire matrix-valued function.

\begin{theorem}\label{Th1.CK}
Let $m>1$ and $R, \sigma_0>0$.
Then, for every $\bu_0\in G_{\sigma_0}^{m}$ with $2\lVert \bu_0 \rVert_{G_{\sigma_0}^{m}} < R$ and ${\sigma_1 \in (0,\sigma_0)}$, there exists an $\eta = \eta(R,m,\sigma_0,\sigma_1) > 0$ such that for $T_0 = (\sigma_0-\sigma_1)/\eta$ there exists a local solution
  $\bu \in C^1\big((0,T_0],G^{m-1}_{\sigma_1}\big) \cap C\big([0,T_0],G^{m}_{\sigma_1}\big)$ to~\eqref{1.PDE}, satisfying
  \begin{align} \sup_{T \in[0,T_0]} \| u(T) \|_{G^{m}_{\sigma_1}} \leq R. \label{solbound}\end{align}
\end{theorem}

The proof of Theorem~\ref{Th1.CK} is standard, cf.~\cite[Theorem~1.1]{SA95}. Full details for the relevant case (with $\bu(X,T) \in \R^2$ and ${\bf M}(\bu)$ as in~\eqref{above}) can be found in~\cite[Theorem~4.2.2]{HA19}. Nevertheless, we provide below a derivation of a priori estimates on $\bu(T)$ in the corresponding Gevrey spaces, because it helps the reader to understand how a  time-dependent scale of Gevrey spaces leads to useful estimates in our main result, Theorem~\ref{mainresult1}. We emphasize that the obtained a priori estimates can be justified a posteriori in a standard way using the vanishing viscosity method. Such an approach is for instance adopted in~\cite[Theorem~3.1]{BKS20}.

Thus, let us derive these a priori estimates assuming we have a sufficiently regular local solution $\bu(T)$ to~\eqref{1.PDE}. Consider the Fourier multiplier operator $|k|_{op}:=\sqrt{-\partial_x^2}$, and multiply equation~\eqref{1.PDE}
by $$e^{2\sigma(T)(1+|k|_{op})}\left(1+|k|_{op}^{2}\right)^m\bu$$
where $\sigma(T) = \sigma_0 - \eta T$.
Integration over $X\in\R$ gives after straightforward calculations that
\begin{align*}
\frac12\frac d{dT}\|\bu\|^2_{G^m_{\sigma(T)}}+\eta&\|(1+|k|_{op})^{1/2}\bu\|^2_{G^m_{\sigma(T)}}\\
&=\mathrm{Re}\left(
(({\bf M}(\bu)-{\bf M}(0))\partial_X\bu,\bu)_{G^m_{\sigma(T)}}+({\bf M}(0)\partial_X\bu,\bu)_{G^m_{\sigma(T)}}
\right).
\end{align*}
The Cauchy-Schwarz inequality
\begin{align}
\mathrm{Re}(\bu,\bv)_{G^m_\sigma}\le\|\bu\|_{G^{m-1/2}_\sigma}\|\bv\|_{G^{m+1/2}_\sigma}\,,
\label{CauchySchwarz}
\end{align}
together with~\eqref{1.alg},~\eqref{1.f} and the assumption $m-\frac12>\frac12$
yield
\begin{equation*}
\frac12\frac d{dT}\|\bu\|^2_{G^m_{\sigma(T)}}+\eta\|\bu\|^2_{G^{m+1/2}_{\sigma(T)}}\le
\|{\bf M}(0)\|\|\bu\|^2_{G^{m+1/2}_{\sigma(T)}}+\phi_m(\|\bu\|_{G^{m-1/2}_{\sigma(T)}})\|\bu\|^2_{G^{m+1/2}_{\sigma(T)}}\,,
\end{equation*}
where  $\phi_m(z)$ is an entire function, which is monotonically increasing on $\R_+$ and satisfies $\phi_m(0)=0$.
Finally, this gives
\begin{equation}\label{1.good}
\frac12\frac d{dT}\|\bu\|^2_{G^m_{\sigma(T)}}+\(\eta-\|{\bf M}(0)\|-\phi_m(\|\bu\|_{G^m_{\sigma(T)}})\)
\|\bu\|^2_{G^{m+1/2}_{\sigma(T)}}\le 0\,.
\end{equation}
Now fix $\eta$ in such a way that $\eta>\|{\bf M}(0)\|+ \phi_m(R)$. Then, the a priori estimate~\eqref{1.good} finally yields~\eqref{solbound}.

\subsection{Approximate solutions for the perturbed problem}

\label{sec-approx}

For approximation theorems it is essential that the residual, i.e., those terms
that remain after inserting the approximation
into the original system, is sufficiently small. In order to obtain a sufficiently small residual,
higher-order terms have to be added to the WMEs approximation~\eqref{whi1b}.
Hence, the first step of the proof of Theorem~\ref{mainresult1} is the construction of an improved approximation.

We proceed as in~\cite{BKS20} and consider an approximation to the solution $(\br,\bu)$ to~\eqref{hann4} of the form
 \begin{align*} 
  \begin{split}
\br(X,T,\eps)&=\br^0(X,T)+\eps \br^1(X,T)+\eps^2\br^2(X,T)+\cdots ,\\
\bu(X,T,\eps)&=\bu^0(X,T)+\eps \bu^1(X,T)+\eps^2\bu^2(X,T)+\cdots.
\end{split}
 \end{align*}
Inserting these expansions into equation~\eqref{hann4} and equating the terms with the
same powers of $\eps$, we get at $\eps^0$ for $\br^0$ and at $\epsilon^1$ for $\bu^0$,
 \begin{align}
0&= {\bf M}_r(\br^0,\bu^0) , \label{ord0} \\
\Dt \bu^0&={\bf M}_u(\br^0,\bu^0)\partial_X (\br^0,\bu^0), \label{ord1}
 \end{align}
 and by inserting the solution $ \br^0 = \br^0(\bu^0) $ to~\eqref{ord0} into~\eqref{ord1}, we obtain
\begin{equation}\label{2.PDE0}
\Dt \bu^0={\bf M}(\bu^0)\partial_X \bu^0\,,\quad\mbox{with}\ \bu^0\big|_{T=0}=\bu_0\,,
\end{equation}
which coincides with equation~\eqref{1.PDE} studied earlier.

The higher-order terms $(\br^n,\bu^n)$,
$n\in \mathbb N$ can be found by solving the following inhomogeneous equations, which arise at $\epsilon^n$ for $\br^n$ and at $\epsilon^{n+1}$ for $\bu^n$ respectively, namely
\begin{align} \label{hann14}
0&= \widetilde{\bf M}_r(\br^n,\bu^n,\br^0,\bu^0) +  {\bf F}_{r,n}(D_X^2\br^0,D_X\bu^0,\ldots, D_X^2\br^{n-1},D_X\bu^{n-1},\partial_T \br^{n-1}), \\
\Dt \bu^n&={\bf M}_u(\br^0,\bu^0)\partial_X (\br^n,\bu^n)
+ D{\bf M}_u(\br^0,\bu^0) [(\br^n,\bu^n)]\partial_X(\br^0,\bu^0) \nonumber 
\\& \qquad+  {\bf F}_{r,n}(D_X^3\br^0,D_X^2\bu^0,\ldots, D_X^3\br^{n-1},D_X^2\bu^{n-1}) \nonumber
\end{align}
where $D{\bf M}_u(\br^0,\bu^0) [(\br^n,\bu^n)]$ denotes the linearization of the map $(r,u) \mapsto {\bf M}_u(r,u)$ in point $(\br^0,\bu^0)$ applied to $(\br^n,\bu^n)$.
By the implicit function theorem, equation~\eqref{hann14} can be solved w.r.t.
\begin{equation*}
 \br^n = \br^n(D_X^2\br^0,D_X\bu^0,\ldots, D_X^2\br^{n-1},D_X\bu^{n-1},\bu^{n},\partial_T \br^{n-1})
\end{equation*}
for  $ (\br^0,\bu^0,\ldots, \br^{n-1},\bu^{n-1},\bu^{n},\partial_T \br^{n-1}) $ sufficiently small and regular.
In particular $\br^n = 0$ for $(D_X^2\br^0,D_X\bu^0,\ldots, D_X^2\br^{n-1},D_X\bu^{n-1},\bu^{n},\partial_T \br^{n-1}) = 0$ and is analytic in a neighborhood of $0$.
Note that we can express $\br^{n-1}$ -- and therefore $\partial_T \br^{n-1}$ and $D_X^2\br^{n-1}$, too -- in terms of $(\br^0,\bu^0,\ldots, \br^{n-2},\bu^{n-2}, \bu^{n-1})$ and its temporal and spatial derivative. Applying this scheme iteratively shows that all $(\br^0,\ldots, \br^{n-1})$ can be expressed in terms of $(\bu^0,\ldots,\bu^{n-1}) $.
In that sense $\br^n$ is completely determined by $\bu^n$.
Therefore we obtain the equation
\begin{equation}\label{2.app}
\Dt \bu^n = \widetilde{\bf M}(\bu^0)\partial_X \bu^n + \widetilde{D{\bf M}}(\bu^0) [\bu^n] \partial_X \bu^0 +{\bf F}_n(D_X^3\br^0,D_X^2\bu^0,\ldots, D_X^3\br^{n-1},D_X^2\bu^{n-1})
\end{equation}
with zero initial data for $n \geq 1$ by the implicit function theorem.
We note that $\widetilde{\bf M}, {\bf F}_n$ as well as the linearization $\widetilde{D{\bf M}}_u(\bu^0)$ are entire functions and ${\bf F}_n(0)=0$.

Local existence of the improved approximations $\bu^n, n \geq 1$, in a time-dependent scale of Gevrey spaces can be established along the lines of Theorem~\ref{Th1.CK} using standard methods~\cite{SA95}. Again, we refer to~\cite{HA19} for a full proof of Theorem~\ref{Th2.sol} below and only derive corresponding a priori estimates. As in the previous section, we emphasize that the a priori estimates can be justified a posteriori using the vanishing viscosity method.

Thus, we assume sufficient regularity on the solutions $\bu^n$ to~\eqref{2.app} for $n \geq 1$ and on the solution $\bu^0$ to~\eqref{2.PDE0}. We proceed as in~\S\ref{sec-CK-gevrey}. We set $\sigma(T) = \sigma_1 - \eta T$ and obtain the basic energy estimate
\begin{equation*}
\frac12\frac d{dT}\|\bu^n\|^2_{G^m_{\sigma(T)}}+
\(\eta-\frac{q}{2}-\|\widetilde{\bf M}(0)\|-\phi_m(\|\bu^0\|_{G^{m+1/2}_{\sigma(T)}})\)
\|\bu^n\|^2_{G^{m+1/2}_{\sigma(T)}}\le \frac{1}{2q}\|{\bf F}_n\|_{G^{m-1/2}_{\sigma(T)}}^2,
\end{equation*}
for each $q > 0$ by multiplying~\eqref{2.app} with $e^{2\sigma(T)(|k|_{op}+1)}(1+|k|_{op}^2)^m\bu^n(T)$, integrating over $X \in \R$ and applying Young's inequality.
Here, $\phi_m$ is an entire function, which is monotonically increasing on $\R_+$, satisfies $\phi_m(0) = 0 $ and depends on $\widetilde{\bf M}$, $\widetilde{D{\bf M}}$ and $m$ only.
If the solvability condition
 \begin{equation*} 
\eta-\|\widetilde{\bf M}(0)\|-\phi_m(\|\bu^0(T)\|_{G^{m+1/2}_{\sigma(T)}})>0,
 \end{equation*}
is satisfied uniformly for $ T \in [0,\min\{T_0,\sigma_1/\eta\}] $, and we take $q > 0$ sufficiently small, we get the recursive estimate
   \begin{equation*}
   \|\bu^n(T)\|_{G^m_{\sigma(T)}}^2\le CT\sup_{\tau \in [0,T]} \|{\bf F}_n\|_{G^{m-1/2}_{\sigma(\tau)}}^2.
   \end{equation*}
   However, since the number of derivatives in ${\bf F}_n$
   grows with $n$, we need to assume $\bu^0 \in C([0,T],G_{\sigma(T)}^{m+p})$ for a $p \in \N$ and all $T \in [0,T_0]$ in order to bound the $G^m_{\sigma(T)}$-norm of ${\bf F}_n$. Another way to bound this norm is to decrease the exponent $\sigma$ slightly and use the estimate
   \begin{equation*} 
   \|\bu\|_{G^{m+p}_{\sigma-\delta}}\le \frac C{\delta^p}\|\bu\|_{G^m_\sigma}\,,\quad \delta,p>0, \sigma \geq 0.
   \end{equation*}
   Thus, we have acquired
    \begin{equation*} 
    \|\bu^n(T)\|_{G^m_{\sigma(T)-\delta}}\le Q_{\delta,n}
    \(\sup_{\tau \in [0,T_0]}\|\bu^0(\tau)\|_{G^{m}_{\sigma(\tau)}}\), \qquad T \in [0,\min\{T_0,(\sigma_1-\delta)/\eta\}],
    \end{equation*}
   for some monotonically increasing function $Q_{\delta,n}$.

   \begin{theorem}\label{Th2.sol} Let $m>1$ and $\sigma_0 > 0$. Suppose there exists a local solution
   \begin{align*}
    \bu^0 \in C^1\big((0, T_0],\Chi^{m-1}_{\sigma_0}\big) \cap C\big([0,T_0],\Chi^{m}_{\sigma_0}\big),
   \end{align*}
to~\eqref{2.PDE0}. Then, for every $\sigma_1 \in (0,\sigma_0)$, $n \in \mathbb N$, and for all $0 < k \leq n $, there exist $T_1 = T_1(\sigma_1,k+1) \leq T_1(\sigma_1,k) \leq T_0$ and solutions
\begin{align*}
 \bu^k \in C^1\big((0,T_1],G^{m-1}_{\sigma_1}\big) \cap C\big([0,T_1],G^{m}_{\sigma_1}\big),
\end{align*}
to~\eqref{2.app}.
\end{theorem}

   We introduce the $n$-th order approximations
\begin{align*} 
\widetilde{\br}^n(T) &=\br^0(T)+\veps \br^1(T)+\cdots +\veps ^n \br^n(T)\,, \\
\widetilde{\bu}^n(T) &=\bu^0(T)+\veps  \bu^1(T)+\cdots +\veps ^n \bu^n(T)\,,
\end{align*}
and the corresponding residuals
\begin{align*} 
\textsf{Res}_r^n(T)&= {\bf M}_r(\widetilde{\br}^n,\widetilde{\bu}^n) +\eps  {\bf F}_r(D^k_X\widetilde{\br}^n,D^k_X\widetilde{\bu}^n,\partial_T \widetilde{\br}^n),
\\
\textsf{Res}_u^n(T)&= - \Dt \widetilde{\bu}^n+ {\bf M}_u(\widetilde{\br}^n,\widetilde{\bu}^n)\partial_X (\widetilde{\br}^n,\widetilde{\bu}^n)+\eps {\bf F}_u(D^k_X\widetilde{\br}^n,D^k_X\widetilde{\bu}^n).
\end{align*}
As a direct consequence of Theorem~\ref{Th2.sol} we find
\begin{corollary}\label{Cor2.res}
Assume that the hypotheses of Theorem~\ref{Th2.sol} are met. Then, for every $n \in \mathbb N$, the approximate solutions $(\widetilde \br^n,\widetilde \bu^n)$ and residuals
  $(\mathsf{Res}_r^n,\mathsf{Res}_u^n)$ are in $C\big([0,T_1],G^{m}_{\tilde{\sigma}_1}\big)$ for all $\tilde{\sigma}_1 \in [0,\sigma_1)$. 
Further, there exists a constant $C > 0$ such that we have
\begin{align*}
 \sup_{T \in [0,T_1]}\lVert (\widetilde{\br}^n(T), \widetilde{\bu}^n(T)) - (\br^0(T), \bu^0(T)) \rVert_{G^{m}_{\tilde \sigma_1}} &\leq C \epsilon, \quad \ \
 \sup_{T \in [0,T_1]} \lVert (\textsf{\upshape Res}_r^n,\textsf{\upshape Res}_u^n)(T) \rVert_{G^m_{\tilde \sigma_1}} \leq C\veps^{n+1}.
\end{align*}

\end{corollary}

For a connection between the lifespan of the solutions to~\eqref{2.PDE0} and~\eqref{2.app} we refer to~\cite[Remark~4.3]{BKS20}.


\section{Exact solutions and error estimates}

\label{secdrei}

In Section~\ref{sec2bks} we constructed approximate solutions $(\widetilde \br^n,\widetilde \bu^n)$, $n \in \mathbb N$. In the associated Gevrey norms, these approximate solutions lie, by Corollary~\ref{Cor2.res}, $\mathcal{O}(\epsilon)$-close to the WMEs approximation $(\br^0,\bu^0)$ satisfying~\eqref{ord0}-\eqref{ord1} or, equivalently,~\eqref{wyk2}-\eqref{wyk4}. In addition, Corollary~\ref{Cor2.res} implies that inserting the approximate solutions in the modulation equation~\eqref{hann1} (or, equivalently,~\eqref{hann4}) yields residuals $(\textsf{Res}_r^n,\textsf{Res}_u^n)$ of order $ \mathcal{O}(\veps^{n+1}) $.

We switch back to $(x,t)$-coordinates, under which system~\eqref{hann1} takes the form~\eqref{hamb1}. All that remains is to estimate the difference between the approximate solutions and exact solutions to~\eqref{hamb1} on the long $ \mathcal{O}(1/\veps) $-time scale, for which we will use a  number of properties of~\eqref{hamb1}. Before we do so we, we take care of the fact that  $ \psi $  is one time less regular than $ r$ and $ B $.
Thus, we  replace $ \psi $ by the new variable $v = \mathcal{M}^{-1} \psi $, where $\mathcal{M}$ is the Fourier multiplier operator defined through
\begin{equation} \label{hamb0}
\widehat{\psi}(k) = \sqrt{1+k^2}\; \widehat{v}(k).
\end{equation}
In the new variables $ r$, $ v $, and $ B $ system~\eqref{hamb1} reads
\begin{align} \label{hambnew1}
 \begin{split}
\partial_t r & = \partial_x^2 r - \alpha \partial_x (\mathcal{M}v) + e^{2\rho}(1- e^{2r}) + (\partial_x r)^2
-2 q  (\mathcal{M}v) -  (\mathcal{M}v)^2
\\& \qquad  - 2 \alpha q (\partial_x r)-2 \alpha  (\mathcal{M}v)(\partial_x r) + \gamma_r B, \\
\partial_t v & = \partial_x^2 v + \alpha \mathcal{M}^{-1}\partial_x^3 r  + \beta e^{2\rho} \mathcal{M}^{-1} \partial_x (1- e^{2r}) + \alpha \mathcal{M}^{-1} \partial_x( (\partial_x r)^2) - 2 \alpha  q \partial_x v\\& \qquad - \alpha \mathcal{M}^{-1} \partial_x( (\mathcal{M}v)^2)  +2 q \mathcal{M}^{-1}\partial_x^2 r
 +2 \mathcal{M}^{-1} \partial_x((\mathcal{M}v)(\partial_x r)) + \gamma_i \mathcal{M}^{-1} \partial_x B,  \\
\partial_t B  & =  a\partial_x^2 B + c \partial_x B + d e^{2\rho} \partial_x (e^{2r}-1),
\end{split}
\end{align}
which we abbreviate as
\begin{equation}\label{webex1}
\partial_t W= \Lambda W +G(W),
\end{equation}
where $ \Lambda W $ stands for the linear and $ G(W)$ for the nonlinear terms. We observe that now all components of $W$ have the same regularity. Counting derivatives and using the properties~\eqref{1.alg} and~\eqref{1.f}, we immediately see that $ G $ is a mapping from $ \Chi^m_\sigma $ to $ \Chi^{m-1}_\sigma $ for $m > 3/2$. In addition, the linearity $\Lambda \colon D(\Lambda) \subset G_\sigma^{m-1} \to G_\sigma^{m-1}$ with $D(\Lambda) = G_\sigma^{m+1}$ is sectorial and, thus, generates an analytic semigroup on $\Chi^{m-1}_\sigma$. Therefore, system~\eqref{webex1} is a semilinear parabolic equation, and local existence and uniqueness of solutions in the Gevrey spaces $\Chi^m_\sigma$ for $m > 3/2$ follows by standard contraction mapping principles.

\subsection{The linear problem}

\label{seclin31}
We study the linear part $\partial_t W = \Lambda W$ of~\eqref{webex1}, which reads
 \begin{align}
 \begin{split}
\partial_t r & =  \partial_x^2 r - \alpha \partial_x (\mathcal{M}v) -2 e^{2\rho} r
-2 q  (\mathcal{M}v)  - 2 \alpha q (\partial_x r) + \gamma_r B , \\
\partial_t v & =  \partial_x^2 v + \alpha \mathcal{M}^{-1}\partial_x^3 r  -2 \beta  e^{2\rho}\mathcal{M}^{-1}\partial_x r - 2 \alpha  q \partial_x v +2 q \mathcal{M}^{-1} \partial_x^2 r
 + \gamma_i \mathcal{M}^{-1} \partial_x B, \\
\partial_t B  & =   a\partial_x^2 B + c \partial_x B +2  d e^{2\rho} \partial_x r.
\end{split}\label{linsys}
\end{align}
This system can be solved explicitly in Fourier space, where it takes the form $\partial_t \widehat{W} = \widehat{\Lambda}(k) \widehat{W}$. For the forthcoming estimation of the error, we obtain pointwise bounds on $\widehat{\Lambda}(k)$ below.

First, we study the matrix $\widehat{\Lambda}(k)$ near frequency $ k= 0 $. Using that the operator $\mathcal M$ corresponds to pointwise multiplication with the function $k \mapsto \sqrt{1+k^2}$ in Fourier space, one readily observes that $\widehat{\Lambda}(k)$ depends continuously on $k$. Moreover, we find
\begin{align*}
\widehat{\Lambda}(0) = \left( \begin{array}{ccc}  -2e^{2\rho} & -2q & \gamma_r \\ 0 & 0 & 0 \\ 0 & 0 & 0 \end{array}
\right).
\end{align*}
So, $-2e^{2\rho}$ is a simple negative eigenvalue of $\widehat{\Lambda}(0)$, whereas $0$ is a semisimple eigenvalue of $\widehat{\Lambda}(0)$ of algebraic and geometric multiplicity $2$.

By its continuous dependence on $k$ the most negative eigenvalue of $\widehat{\Lambda}(k)$ can be separated from the rest of the spectrum for $k$ sufficiently small. More specifically, there exist constants $\delta_\Lambda \in (0,1)$ and $c_{\Lambda,s} > 0$ such that for $|k| < \delta_\Lambda$ the set $\sigma(\widehat{\Lambda}(k)) \cap \{\lambda \in \C : \mathrm{Re}(\lambda) \leq -c_{\Lambda,s}\}$ consists of a simple eigenvalue $\lambda_1(k)$ only.

Since $\widehat{\Lambda}(k)$ is analytic in $k$ on the disk of radius $\delta_\Lambda$ centered at the origin, standard analytic perturbation theory~\cite{KA95} yields that the spectral projection $P_1(k)$ of $\widehat{\Lambda}(k)$ onto the eigenspace associated with $\lambda_1(k)$ is also analytic in $k$ on that disk. For $|k| < \delta_\Lambda$ it holds
\begin{align} \widehat{\Lambda}(k) P_1(k) = \lambda_1(k) P_1(k), \qquad \widehat{\Lambda}(0) (1-P_1(0)) = 0, \label{star2}\end{align}
where the last equality follows from the fact that $0$ is a semisimple eigenvalue of $\widehat{\Lambda}(0)$.

We now introduce
\begin{align}
 P_c(k) = \chi_\Lambda(k)(1-P_1(k)), \qquad P_s(k) = 1-P_c(k), \label{specdecomp}
\end{align}
for $k \in \R$, where $\chi_\Lambda \colon \R \to [0,1]$ is a smooth cut-off function whose support lies in $(-\delta_\Lambda,\delta_\Lambda)$, and whose value is $1$ on the interval $[-\delta_\Lambda/2,\delta_\Lambda/2]$. Subsequently, we decompose
$$\widehat{\Lambda}(k) = \widehat{\Lambda}_s(k) + \widehat{\Lambda}_c(k), \qquad \widehat{\Lambda}_j(k) := \widehat{\Lambda}(k)P_j(k), \qquad j = s,c.$$
We note that $P_s(k)$ and $P_c(k)$, and thus $\widehat{\Lambda}_s(k)$ and $\widehat{\Lambda}_c(k)$, are smooth in $k \in \R$. Hence, by the Cauchy-Schwarz inequality,~\eqref{star2} and~\eqref{specdecomp}, there exist constants $c_{\Lambda,c}, c_{\Lambda,s} > 0$ such that for $k \in [-\delta_\Lambda/2,\delta_\Lambda/2]$ we have the bounds
\begin{align}
\begin{split}
 \mathrm{Re} \, \langle \widehat V_s(k), \widehat{\Lambda}_s(k) \widehat V_s(k)\rangle &= \mathrm{Re}\,\lambda_1(k) \lvert \widehat V_s(k) \rvert^2 \leq -c_{\Lambda,s}\lvert \widehat V_s(k) \rvert^2,\\
 \mathrm{Re} \, \langle \widehat V_c(k), \widehat{\Lambda}_c(k) \widehat V_c(k) \rangle &\leq \|\widehat{\Lambda}_c(k)\|\lvert \widehat V_c(k) \rvert^2 \leq c_{\Lambda,c}|k|\lvert \widehat V_c(k) \rvert^2,
 \end{split}
 \label{lowfreq}
\end{align}
for $\widehat V_j(k) = P_j(k) \widehat V$, $j = s,c$ with $\widehat{V} \in \mathbb{C}^3 $.

On the other hand, to estimate $\widehat{\Lambda}(k)$ outside of the neighborhood $(-\delta_\Lambda,\delta_\Lambda)$ of $0$, we exploit that the linear system~\eqref{linsys} is parabolic. In particular, there exist positive constants $\tilde{c}_{\Lambda,1}$ and $\tilde{c}_{\Lambda,2}$ such that for $k \in \R \setminus  (-\delta_\Lambda,\delta_\Lambda)$ we have the bound
\begin{align}
 \mathrm{Re} \, \langle \widehat V, \widehat{\Lambda}(k) \widehat V\rangle \leq \tilde{c}_{\Lambda,\infty} |k|\lvert \widehat V \rvert^2 - \tilde{c}_{\Lambda,2} k^2\lvert \widehat V \rvert^2. \label{highfreq}
\end{align}
for vectors $V \in \mathbb{C}^3$.

We recall that $\chi_\Lambda(k)$ vanishes on $\R \setminus (-\delta_\Lambda,\delta_\Lambda)$ and thus $P_s(k) = 1$ and $P_c(k) = 0$ for $k \in \R \setminus (-\delta_\Lambda,\delta_\Lambda)$ by~\eqref{specdecomp}. Hence, by combining the low and high frequency bounds~\eqref{lowfreq} and~\eqref{highfreq}, respectively, and using that $\widehat{\Lambda}_s(k)$ and $\widehat{\Lambda}_c(k)$ are continuous in $k$ to bound the middle frequencies, we obtain constants $d_{\Lambda,0},d_{\Lambda,1},d_{\Lambda,2} > 0$ such that the following pointwise estimates hold
\begin{align}
\label{webex4}
\begin{split}
\textrm{Re}  \, \langle \widehat{V}_s(k), \widehat{\Lambda}_s(k) \widehat{V}_s(k) \rangle &\leq (- d_{\Lambda,0} + d_{\Lambda,1} |k| - d_{\Lambda,2} k^2) \lvert \widehat V_s(k) \rvert^2,\\
\textrm{Re} \, \langle \widehat{V}_c(k), \widehat{\Lambda}_c(k) \widehat{V}_c(k) \rangle &\leq  d_{\Lambda,1} |k|\lvert \widehat V_c(k) \rvert^2,
\end{split}
\end{align}
for $k \in \R$ and $\widehat V_j(k) = P_j(k) \widehat V$, $j = s,c$ with $\widehat{V} \in \mathbb{C}^3 $.

\begin{remark}{\upshape \label{remhypel}
The action of $\widehat{\Lambda}_c(k)$ on the subspace spanned by $P_c(k)$ is analytic in $k$ for $|k| < \delta_{\Lambda/2}$.
As observed above, it holds $\widehat{\Lambda}_c(0) = 0$.
Since $0$ is a semisimple eigenvalue, the first order expansion with respect to $\mathrm{i}k$, corresponding to the long-wave limit of the system reduced to the subspace, is given by $\widehat{\Lambda}_c'(0) = P_c(0) \widehat{\Lambda}'(0) P_c(0)$, see~\cite[II.2.2]{KA95}.
If we diagonalize $\widehat{\Lambda}(0)$ with respect to the basis $\{ (1,0,0), (-q e^{-2\rho}, 1,0), (\gamma_r e^{-2\rho}/2, 0,1)\}$ of its eigenvectors and call the transform $S$, we find the transformed matrix
\begin{align*}
S^{-1}\widehat{\Lambda}_c'(0)S =
 \begin{pmatrix}
  0 & 0 \\
  0 & \Lambda_2'(0)
 \end{pmatrix}, \qquad \text{ where } \qquad
 \widehat{\Lambda}_2'(0) =
 \begin{pmatrix}
  2 q   (  \beta - \alpha) & \gamma_i - \beta  \gamma_r\\
  - 2d q  & c+ d\gamma_r
 \end{pmatrix}.
\end{align*}
 Note that $\widehat{\Lambda}_2'(0)\partial_X$ is the linearization of the WMEs~\eqref{whi1b}. The matrix $\widehat{\Lambda}'_2(0)$ either has two distinct real eigenvalues, corresponding
to the hyperbolic situation, or two complex conjugate eigenvalues, which corresponds to the elliptic situation.
}\end{remark}

\begin{remark} {\upshape
The possible linear growth rates of the semigroup associated with~\eqref{linsys} due to~\eqref{webex4} can be controlled by working in a time-dependent scale
of Gevrey spaces. This corresponds to the introduction of time-dependent new variables
$$
\widehat{W}_s(k,t) = e^{- (\sigma_0/\veps- \eta t)|k|} \widehat{Z}_s(k,t), \qquad \widehat{W}_c(k,t) = e^{- (\sigma_0/\veps- \eta t)|k|} \widehat{Z}_c(k,t).
$$
The linearities in the $\widehat{Z}_s$- and $\widehat{Z}_c$-equations are then given by $\widetilde{\Lambda}_s(k) = \widehat{\Lambda}_s(k) - \eta|k|$ and $\widetilde{\Lambda}_c(k) = \widehat{\Lambda}_c(k) - \eta|k|$, respectively. By choosing $ \eta > 0 $ sufficiently large, we obtain the estimates
\begin{align*}
\begin{split}
\textrm{Re}  \, \langle \widehat{Z}_s(k), \widehat{\Lambda}_s(k) \widehat{Z}_s(k) \rangle &\leq (- d_{\Lambda,0} - \eta |k|/2 - d_{\Lambda,2} k^2) \lvert \widehat Z_s(k) \rvert^2,\\
\textrm{Re} \langle \widehat{Z}_c(k), \widehat{\Lambda}_c(k) \widehat{Z}_c(k) \rangle &\leq  -\frac{\eta |k|}{2} \lvert \widehat Z_c(k) \rvert^2,
\end{split}
\end{align*}
for $k \in \R$ and $\widehat Z_j(k) = P_j(k) \widehat Z$, $j = s,c$ with $  \widehat{Z} \in \mathbb{C}^3 $. Hence, by working  in a time-dependent scale of Gevrey spaces linear damping for system~\eqref{webex1} can be obtained.}
\end{remark}

\subsection{Decomposition in center and stable modes}

To exploit the linear growth bounds derived in~\S\ref{seclin31}, we decompose the variable $W$ in~\eqref{webex1} in accordance with the spectral decomposition of $\widehat{\Lambda}(k)$ as
\begin{align}W = W_s + W_c, \label{decomp} \end{align}
where $W_s$ and $W_c$ are defined by their Fourier transforms:
$$\widehat{W}_s(k) = P_s(k)\widehat{W}(k), \qquad \widehat{W}_c(k) = P_c(k)\widehat{W}(k).$$
Since $W$ satisfies~\eqref{webex1}, the resulting system for $W_s$ and $W_c$ reads in Fourier space:
\begin{align} \label{adamoF}
 \begin{split}
 \partial_t \widehat{W}_s & = \widehat{\Lambda}_s \widehat{W}_s + \widehat{g}_s(\widehat{W}), \\
\partial_t \widehat{W}_{c} & = \widehat{\Lambda}_c \widehat{W}_{c} + \widehat{g}_c(\widehat{W}).
\end{split}
\end{align}
where we denote
\begin{align*} \widehat{g}_j(\widehat{W})= P_j(k) \mathcal{F} G\left(\mathcal{F}^{-1}\widehat{W}\right), \qquad j = s,c.\end{align*}

The nonlinearity in system~\eqref{webex1} is of the form
\begin{equation*} 
G(W) = \begin{pmatrix} f_r(W) \\ \partial_x f_v(W) \\ \partial_x f_B(W)\end{pmatrix},
\end{equation*}
i.e., in front of all terms on the right-hand side of the $ v$- and $ B $-equation in~\eqref{hambnew1} there is a derivative. This structure implies that the nonlinearity in the $\widehat{W}_c$-equation in~\eqref{adamoF} vanishes at frequency $k = 0$. Indeed, one readily computes
\begin{align*}P_c(0)\begin{pmatrix} * \\ 0 \\ 0\end{pmatrix} = \left( \begin{array}{ccc}  0 & -\frac{q}{e^{2\rho}} & \frac{\gamma_r}{2e^{2\rho}} \\ 0 & 1 & 0 \\ 0 & 0 & 1 \end{array}\right)\begin{pmatrix} * \\ 0 \\ 0\end{pmatrix} = 0.
\end{align*}
Hence, since $P_c(k)$ is analytic in $k$ on a disk of radius $\delta_\Lambda/2$ centered at $k = 0$, the nonlinearity in the $\widehat{W}_c$-equation in~\eqref{adamoF} can be written as $$\widehat{g}_c(\widehat{W}) = P_c \mathcal{F} G\left(\mathcal{F}^{-1}\widehat{W}\right) = \widehat{\vartheta} \breve{g}_c(\widehat{W}),$$
for $|k| < \delta_\Lambda/2$ with
\begin{align*}
\widehat{\vartheta}(k) = \chi_\Lambda(k) k, \qquad \breve{g}_c(\widehat{W}) = \mathrm{i} P_c(0)\mathcal{F}\begin{pmatrix} 0 \\ f_v(\mathcal{F}^{-1}\widehat{W}) \\ f_B(\mathcal{F}^{-1}\widehat{W}))\end{pmatrix} + \frac{P_c(k)-P_c(0)}{k}\mathcal{F}G(\mathcal{F}^{-1}\widehat{W}).
\end{align*}
Thus, we denote~\eqref{adamoF} in physical space as
\begin{align} \label{adamo}
\begin{split}
\partial_t  W_s & = \Lambda_s W_s +  g_s(W) , \\
\partial_t  W_{c} & =  \Lambda_{c}  W_{c} +   \vartheta g_{c}(W)  .
\end{split}
\end{align}
where $\Lambda_s$, $\Lambda_c$ and $\vartheta$ are Fourier multiplier operators corresponding to $\widehat{\Lambda}_s$, $\widehat{\Lambda}_c$ and $\widehat{\vartheta}$, respectively. Moreover, the nonlinear mappings $ g_s $ and $ g_c $ are defined by
$$g_s(W) = \mathcal{F}^{-1} \widehat{g}_s(\widehat{W}), \qquad g_c(W) = \mathcal{F}^{-1} \breve{g}_c(\widehat{W}).$$
Since the nonlinearity $G$ in~\eqref{webex1} is a mapping from $\Chi^m_\sigma $ to $ \Chi^{m-1}_\sigma$ for $m > 3/2$, so are $g_s$ and $g_c$ (where we use that, due to analyticity in $|k| < \delta_\Lambda/2$, $k \mapsto (P_c(k)-P_c(0))/k$ is bounded on $\R$).

Thus, on the one hand, the $ W_s $-equation in~\eqref{adamo} is linearly exponentially damped. On the other hand, the $ W_{c} $-equation in~\eqref{adamo}  is not linearly exponentially damped at $ k = 0 $, cf.~\eqref{webex4}, but has a `derivative' $ \vartheta $, with $ |\widehat{\vartheta}(k)| \leq C \min\{1,| k |\} $, in front of its nonlinearity $ g_{c} $.

\subsection{The equations for the error}

We write a solution $W$ to~\eqref{webex1} as sum of the transformed $n$-th order approximation
$ W_{\rm an}$, established in Corollary~\ref{Cor2.res}, and the transformed error $ \veps^\kappa R$, i.e.,
\begin{align}\label{deferror}
W =  W_{\rm an}+  \veps^\kappa R, \qquad W_{\rm an}(0) = W(0), \qquad \kappa > 1,
\end{align}
where we note that the scaling of the error with $\veps^\kappa$ is beneficial for the final energy estimates, cf.~\eqref{assump15.1} and~\eqref{assump17.1}. As in~\eqref{decomp}, we decompose the error as
\begin{align} \label{Rdecomp} R = R_s + R_c,\end{align}
where $R_s$ and $R_c$ are defined by their Fourier transforms:
$$\widehat{R}_s(k) = P_s(k)\widehat{R}(k), \qquad \widehat{R}_c(k) = P_c(k)\widehat{R}(k).$$
By~\eqref{adamo} the equations for $R_s$ and $R_c$ read
\begin{align} \label{webex2}
\partial_t  R_s & = \Lambda_s R_s +  h_s( W_{\rm an}, R)  + \veps^{-\kappa} {\textrm{Res}}_s(W_{\rm an}), \\
\partial_t  R_{c} & =  \Lambda_{c}  R_{c} +   \vartheta h_{c}( W_{\rm an}, R)  + \veps^{-\kappa} {\textrm{Res}}_{c}(W_{\rm an}),
\label{webex3}
\end{align}
where $ \mathop{\mathrm{Res}}_{s,c} ( W_{\rm an}) $ is the transformed residual and where we denote
$$
h_j( W_{\rm an}, R) =  \veps^{-\kappa} \left( g_j \left( W_{\rm an} + \veps^\kappa R\right)-g_j( W_{\rm an}) \right)
$$
for $ j = c,s $. We have the following tame estimate on $h_j(W_{\rm an},R)$.

\begin{lemma} \label{webex44}
Let $m > 3/2$ and $\sigma \geq 0$. For all $ M > 0$, there exists a constant $ C_1(W_{\rm an}) > 0 $, which depends on $\|W_{\rm an}\|_{\Chi^{m+1/2}_\sigma}$ only and satisfies $ C_1(W_{\rm an}) \to  0 $ as $ W_{\rm an} \to 0 $ in $\Chi^{m+1/2}_\sigma$, and there exists a constant $ C_2(M,W_{\rm an}) >0 $, which depends on $M$ and $\lVert W_{\rm an} \rVert_{\Chi^{m+1/2}_\sigma}$ only, such that for all
$ \varepsilon \in (0,1) $ and $R \in \Chi^{m+1/2}_\sigma $ satisfying  $  \| R \|_{ \Chi^{m}_\sigma}  \leq M $  we have the estimate
\begin{equation*}
\| h_j( W_{\rm an}, R) \|_{\Chi^{m-1/2}_\sigma} \leq C_1(W_{\rm an})  \| R \|_{ \Chi^{m+1/2}_\sigma}
+ C_2(M,W_{\rm an}) \veps^\kappa \| R \|_{ \Chi^{m+1/2}_\sigma},
\end{equation*}
for $j = s,c$.
\end{lemma}
\begin{proof}
The nonlinear terms appearing on the right-hand side of~\eqref{hamb1}
are entire functions of  $ r,\psi,B $, and $  \partial_x r $
in the $ r $-equation, of $ r,\psi, \partial_x B, \partial_x r , \partial_x \psi  $, and $  \partial_x^2 r $
in the $ \psi $-equation, and of $ r $ and $  \partial_x r $ in the $ B $-equation.
After the transformation~\eqref{hamb0}, the nonlinear terms appearing on the right-hand side of~\eqref{hambnew1}
are entire functions of  $ r,\mathcal{M} v,B $, and $  \partial_x r $
in the $ r $-equation and of $ r $ and $  \partial_x r $ in the $ B $-equation.
The nonlinear terms appearing on the right-hand side of the $ v $-equation
consist of the Fourier multiplier operator $ \mathcal{M}^{-1} \partial_x$, which corresponds to multiplication with the bounded function $k \mapsto \mathrm{i}k/\sqrt{1+k^2}$, acting on entire functions of $r, \mathcal{M} v, B, \partial_x B$ and $\partial_x r$. Subsequently, these terms are mixed in the nonlinearities $g_s$ and $g_c$ in the equations~\eqref{adamo} for $W_s$ and $W_c$ by pointwise  multiplication in Fourier space with a bounded matrix function. Therefore, $g_s$ and $g_c$ consist of a bounded linear transformation acting on entire functions of $r,\partial_x r, \mathcal{M} v, B$ and $\partial_x B$. Hence, the same holds for $ h_s $ and $ h_c $ upon replacing $ W $ by $ W_{\rm an} + \varepsilon^{\kappa} R $.

Thus, by Corollary~\ref{Cor2.res}, by the fact that $G^{m-1/2}_\sigma$ is an algebra satisfying~\eqref{1.alg} and~\eqref{1.f}, and by the fact that the nonlinear terms in $G$ are (up to a bounded multiplier operator) entire functions of the variables $r, \mathcal{M} v, B, \partial_x B$ and $ \partial_x r$, it is readily seen that the terms in the nonlinearity $h_j(W_{\rm an},R)$ that are linear in $R$ (and thus are paired with a $W_{\rm an}$-contribution) are bounded by $C_1(W_{\rm an})  \| R \|_{ \Chi^{m+1/2}_\sigma}$.

Let us provide the details why the terms in the nonlinearity $h_j(W_{\rm an},R)$ that are nonlinear in $R$ can be bounded by $\smash{C_2(M,W_{\rm an}) \veps^\kappa  \| R \|_{ \Chi^{m+1/2}_\sigma}}$. First, most nonlinear terms in $G$ contain at most one derivative. Hence, using the properties~\eqref{1.alg} and~\eqref{1.f} again and the fact that the nonlinear terms in $G$ are (up to a bounded multiplier operator) entire functions of the variables $r, \mathcal{M} v, B, \partial_x B$ and $ \partial_x r$, the corresponding terms in $h_j(W_{\rm an},R)$ that are nonlinear in $R$ can be bounded by $C_2(M,W_{\rm an}) \veps^\kappa \| R \|_{ \Chi^{m+1/2}_\sigma}$ as long as $\| R \|_{ \Chi^{m}_\sigma} \leq M$. Second, nonlinear terms in $G$ that contain more than one derivative are all quadratic and (up to the bounded multiplier operator $\mathcal M^{-1}\partial_x$) of the form $(\partial_x r)^2$, $(\mathcal M v)^2$ or $(\mathcal M v) (\partial_x r)$. The corresponding terms in $h_j(W_{\rm an},R)$ that are nonlinear in $R$ can thus be bounded with the aid of~\eqref{1.kp}, where we take $m_1 = m-1/2$ and $m_2 = m-1$. This leads again to the bound $C_2(M,W_{\rm an}) \veps^\kappa  \| R \|_{ \Chi^{m+1/2}_\sigma}$ as long as $\| R \|_{ \Chi^{m}_\sigma} \leq M$.
\end{proof}

We emphasize that for our purposes the estimate in Lemma~\ref{webex44} on $ h_c $ is not sufficient, and we need to exploit the factor $\vartheta$ in front of $h_c$ in~\eqref{webex3}, see~Lemma~\ref{Lem3.1}.

\subsection{The energy inequalities}
\label{sec34new}

As outlined in~\S\ref{seclin31}, the fact that the linear part in the equations~\eqref{webex2}-\eqref{webex3} for the error $ R $ can exhibit $ \mathcal{O}(1)$ growth rates in time as $\veps \to 0$ is a serious
problem to obtain estimates on an $ \mathcal{O}(1/\veps) $-time scale. Moreover,
the smallness of the nonlinear terms in the derivation of the WMEs
comes from the derivatives, i.e., from a loss of regularity.
Both problems can be resolved by working in a time-dependent scale of Gevrey spaces as already used
in Section~\ref{sec2bks}.

In contrast to Section~\ref{sec2bks}, we work now in the unscaled $(x,t)$-variables. We take $0 < \tilde{\sigma}_1 < \sigma_1 < \sigma_0$. Then, Theorem~\ref{Th2.sol} yields local existence of the $n$-th order approximation in Gevrey spaces, i.e.,~there exists an $\veps$-independent $T_1 \in (0,\tilde{\sigma}_1/\eta)$ such that the transformed $n$-th order approximation $W_{\mathrm{an}}$ satisfies
\begin{align*}W_{\rm an} \in C^1\big((0,T_1/\veps],G^m_{\sigma_1/\varepsilon}\big) \cap C\big([0,T_1/\veps],G^{m+1}_{\sigma_1/\varepsilon}\big).\end{align*}
Setting $\sigma(t)=\tilde{\sigma}_1/\varepsilon-\eta t$, it follows in particular that
\begin{align}W_{\rm an} \in C^1\big((0,t],G^m_{\sigma(t)}\big) \cap C\big([0,t],G^{m+2}_{\sigma(t)}\big),\label{locexappr}\end{align}
for any $t \in (0,T_1/\veps]$. In addition, local existence of the exact solution $W$ to the semilinear parabolic system~\eqref{webex1} follows by a standard contraction mapping argument, see Remark~\ref{rem:local}. Since the error $\veps^\kappa R$ is defined in~\eqref{deferror} as the difference between the exact solution $W$ and the $n$-th order approximation $W_{\mathrm{an}}$, this yields
\begin{align}R \in C^1\big((0,t],G^m_{\sigma(t)}\big) \cap C\big([0,t],G^{m+2}_{\sigma(t)}\big),  \label{regR}
\end{align}
for $t \in (0,T_1/\veps]$ as long as $\|R(t)\|_{G^m_{\sigma(t)}}$ remains bounded.

Let $M > 0$ be a (sufficiently large) $t$-independent constant, which we will fix a posteriori. Take $t \in (0,T_1/\veps]$ such that
\begin{align} {\sup_{\tau \in [0,t]} \|R(\tau)\|_{G^m_{\sigma(\tau)}}} \leq M. \label{bounded}\end{align}
As in~\S\ref{sec-CK-gevrey} and~\S\ref{sec-approx}, let $|k|_{op}=\sqrt{-\partial_x^2}$, and multiply equations~\eqref{webex2} and~\eqref{webex3} by $e^{2\sigma(t)|k|_{op}}(1+|k|_{op}^{2})^mR_s$ and $e^{2\sigma(t)|k|_{op}}(1+|k|_{op}^{2})^rR_c$, respectively. Integration over $x\in\R$ leads to
\begin{align}\label{star4}
 \begin{split}
  \frac12\frac d{dt}\|R_s\|^2_{G^m_{\sigma(t)}} &\!= -\eta\||k|_{op}^{1/2} R_s\|^2_{G^m_{\sigma(t)}}
 + \mathrm{Re}
( \Lambda_s R_s +  h_s( W_{\rm an}, R)  + \veps^{-\kappa} {\textrm{Res}}_s(W_{\rm an}),R_s)_{G^m_{\sigma(t)}},
\end{split}\\
\begin{split}
 \frac12\frac d{dt}\|R_c\|^2_{G^r_{\sigma(t)}}
&\!= - \eta\||k|_{op}^{1/2} R_c\|^2_{G^r_{\sigma(t)}} + \mathrm{Re}
( \Lambda_c R_c + \vartheta h_c( W_{\rm an}, R) + \veps^{-\kappa} {\textrm{Res}}_c(W_{\rm an}), R_c)_{G^r_{\sigma(t)}}.\end{split} \label{star6}
\end{align}
Here, we emphasize that we can estimate the $G^{m}_{\sigma(t)}$-norm of $R_c$ by its $G^r_{\sigma(t)} $-norm for any $r \geq 0$, because $R_c$ has compact support in Fourier space by~\eqref{specdecomp} and~\eqref{Rdecomp}.

Thus, the linear terms in~\eqref{webex2} and~\eqref{webex3} yield the contributions
$$-\eta\||k|_{op}^{1/2} R_s\|^2_{G^m_{\sigma(t)}} + \mathrm{Re}(\Lambda_s R_s,R_s)_{G^m_{\sigma(t)}},  \qquad -\eta\||k|_{op}^{1/2} R_c\|^2_{G^m_{\sigma(t)}} + \mathrm{Re}(\Lambda_c R_c,R_c)_{G^r_{\sigma(t)}}$$
in~\eqref{star4} and~\eqref{star6}, respectively, which provide damping in $R_s$ and $R_c$, see estimate~\eqref{estF1} below. However, this damping is insufficient to control all $R_c$-contributions in $\smash{(h_s(W_{\rm an}, R),R_s)_{G^m_{\sigma(t)}}}$, see the upcoming estimate~\eqref{estF1} and Lemma~\ref{webex44}. We emphasize that such terms do not arise in $\smash{(\vartheta h_c( W_{\rm an}, R), R_c)_{G^r_{\sigma(t)}}}$, since $\widehat{\vartheta}(k)$ vanishes at $k = 0$, see Lemma~\ref{Lem3.1} below. For the same reason such terms do also not appear in $\smash{(|k|^{1/2}_{op} h_s( W_{\rm an}, R),|k|^{1/2}_{op}R_s)}$. Therefore, we set $r = m-1/2$, and we do not only consider the energies $\smash{\lVert R_s \rVert^2_{G^m_{\sigma(t)}}}$ and $\smash{\lVert R_c\rVert^2_{G^r_{\sigma(t)}}}$, but also $\smash{\lVert |k|^{1/2}_{op}R_s\rVert^2_{G^r_{\sigma(t)}}}$.

Thus, multiplying~\eqref{webex2} by $e^{2\sigma(t)|k|_{op}}(1+|k|_{op}^{2})^r |k|_{op}R_s$ yields
\begin{align}
\begin{split}
  \frac12\frac d{dt}\||k|_{op}^{1/2} R_s\|^2_{G^r_{\sigma(t)}}
& = - \eta\||k|_{op} R_s\|^2_{G^r_{\sigma(t)}} + \mathrm{Re}
( \Lambda_s |k|_{op}^{1/2} R_s,|k|^{1/2}_{op} R_s)_{G^r_{\sigma(t)}}\\
& \qquad + \mathrm{Re}(|k|_{op}^{1/2} h_s( W_{\rm an}, R) + |k|_{op}^{1/2}\veps^{-\kappa} {\textrm{Res}}_s(W_{\rm an}), |k|_{op}^{1/2} R_s)_{G^r_{\sigma(t)}}.\end{split} \label{star5}
\end{align}
We readily establish bounds on most terms on the right-hand side of~\eqref{star4},~\eqref{star6} and~\eqref{star5}.
\begin{enumerate}
\item Using~\eqref{webex4} we bound the linear terms as
\begin{align}
 \begin{split} \label{estF1}
\mathrm{Re}( \Lambda_s R_s ,R_s)_{G^m_{\sigma(t)}}  & \leq - d_{\Lambda,0}
( R_s , R_s)_{G^m_{\sigma(t)}}
+ d_{\Lambda,1} ( |k|_{op}^{1/2} R_s ,|k|_{op}^{1/2} R_s)_{G^{m}_{\sigma(t)}}\\
\mathrm{Re}( \Lambda_s |k|_{op}^{1/2} R_s , |k|_{op}^{1/2} R_s)_{G^r_{\sigma(t)}}  & \leq -d_{\Lambda,0}( |k|_{op}^{1/2} R_s ,|k|_{op}^{1/2} R_s)_{G^{r}_{\sigma(t)}}\\ &\qquad +  d_{\Lambda,1} ( |k|_{op} R_s ,|k|_{op} R_s)_{G^{r}_{\sigma(t)}},\\
\mathrm{Re}( \Lambda_{c}  R_{c} ,R_c)_{G^r_{\sigma(t)}}  & \leq  d_{\Lambda,1} ( |k|_{op}^{1/2}R_c ,|k|_{op}^{1/2} R_c)_{G^{r}_{\sigma(t)}},
\end{split}
\end{align}
with strictly positive constants $  d_{\Lambda,0} $ and $  d_{\Lambda,1} $.
\item
Using~\eqref{CauchySchwarz},~\eqref{bounded} and Lemma~\ref{webex44} we get
\begin{align}
  \begin{split} \label{estF2}
|(  h_s( W_{\rm an}, R)  ,R_s)_{G^m_{\sigma(t)}} |& \leq  \| h_s( W_{\rm an}, R)  \|_{ \Chi^{m-1/2}_{\sigma(t)}}  \| R \|_{ \Chi^{m+1/2}_{\sigma(t)}}
 \\ &\leq  C_1(W_{\rm an})  \| R \|_{ \Chi^{m+1/2}_{\sigma(t)}}^2
+ C_2(M,W_{\rm an}) \veps^\kappa  \| R \|_{ \Chi^{m+1/2}_{\sigma(t)}}^2.
\end{split}
\end{align}
\item
Using Corollary~\ref{Cor2.res} we find
\begin{align}
  \begin{split} \label{estF3}
|( \veps^{-\kappa} {\textrm{Res}}_s(W_{\rm an}),R_s)_{G^m_{\sigma(t)}} |
& \leq \| \veps^{-\kappa} {\textrm{Res}}_s(W_{\rm an}) \|_{ \Chi^{m}_{\sigma(t)}} \| R \|_{ \Chi^{m}_{\sigma(t)}} \\
&\leq \veps^{n+1-\kappa}C_{res} \| R \|_{ \Chi^{m}_{\sigma(t)}}
\leq \veps^{n+1-\kappa}C_{res} (1+\| R \|_{ \Chi^{m}_{\sigma(t)}}^2),
\end{split}
\end{align}
Similarly, using $r = m-1/2$, we obtain
\begin{align}
  \begin{split} \label{estF4}
|(|k|^{1/2}_{op} \veps^{-\kappa} {\textrm{Res}}_s(W_{\rm an}),|k|^{1/2}_{op} R_s)_{G^r_{\sigma(t)}} |
& \leq \| \veps^{-\kappa} {\textrm{Res}}_s(W_{\rm an}) \|_{ \Chi^{m}_{\sigma(t)}} \| |k|^{1/2}_{op} R \|_{ \Chi^{r}_{\sigma(t)}}\\ &\leq \veps^{n+1-\kappa}C_{res} (1+\| |k|^{1/2}_{op} R_s \|_{ \Chi^{r}_{\sigma(t)}}^2),\\
|(\veps^{-\kappa} {\textrm{Res}}_c(W_{\rm an}),R_c)_{G^r_{\sigma(t)}} | & \leq \| \veps^{-\kappa} {\textrm{Res}}_c(W_{\rm an}) \|_{ \Chi^{m}_{\sigma(t)}} \| R_c \|_{ \Chi^{r}_{\sigma(t)}}\\ &\leq \veps^{n+1-\kappa}C_{res} (1+\| R_c \|_{ \Chi^{r}_{\sigma(t)}}^2).
\end{split}
\end{align}
\end{enumerate}
Obtaining useful estimates on the remaining terms
\begin{align*}
(    \vartheta h_{c}( W_{\rm an}, R) ,R_c)_{G^r_{\sigma(t)}}, \qquad (|k|_{op}^{1/2} h_{s}( W_{\rm an}, R) , |k|^{1/2}_{op} R_s)_{G^r_{\sigma(t)}},
\end{align*}
is less trivial. We illustrate how to estimate these terms for the prototypical example $ (R_1 \partial_x( W_{\rm an} R_2))_{L^2} $, which we bound in
Fourier space by
\begin{align*}
| ( R_1, \partial_x (W_{\rm an} R_2)_{L^2} |  &\leq \lVert |k|^{1/2} \widehat R_1 \rVert_{L^2}\lVert |k|^{1/2} (\widehat W_{\rm an} \ast \widehat R_2) \rVert_{L^2}  \\
& \leq \lVert |k|^{1/2} \widehat R_1 \rVert_{L^2} \left( \lVert |k|^{1/2} \widehat W_{\rm an} \ast \widehat R_2 \rVert_{L^2} +  \lVert \widehat W_{\rm an} \ast|k|^{1/2} \widehat R_2 \rVert_{L^2}\right) \\
& \leq \lVert |k|^{1/2} \widehat R_1 \rVert_{L^2} \left( \lVert |k|^{1/2} \widehat{W}_{\rm an} \rVert_{L^1} \lVert \widehat R_2 \rVert_{L^2} +  \lVert \widehat{W}_{\rm an} \rVert_{L^1} \rVert |k|^{1/2} \widehat R_2 \rVert_{L^2}\right) \\
&\leq \frac{1}{2}\lVert |k|^{1/2} \widehat R_1 \rVert_{L^2}^2 + \frac{1}{2} \left( \lVert |k|^{1/2} \widehat{W}_{\rm an} \rVert_{L^1} \lVert \widehat R_2 \rVert_{L^2}\right)^2 \\\ &\quad +  \lVert |k|^{1/2} \widehat R_1 \rVert_{L^2} \lVert \widehat W_{\rm an} \rVert_{L^1} \rVert |k|^{1/2} \widehat R_2 \rVert_{L^2},
\end{align*}
where we used $ \sqrt{|k|} \leq \sqrt{|k-l|} + \sqrt{|l|} $.
Note that $  \|  |k|^{1/2}  \widehat{W}_{\rm an} \|_{L^1}  = \mathcal{O}(\varepsilon^{1/2}) $ due to the long-wave character of $ W_{\rm an} $ in physical space. With this idea in mind, we prove the following result.
\begin{lemma} \label{Lem3.1}
Let $r>1$ and $\sigma \geq 0$. For all $ M > 0$, there exists a constant $ C_3(W_{\rm an}) > 0 $, which depends on $\|W_{\rm an}\|_{\Chi^{r+1}_\sigma}$ only and satisfies $ C_3(W_{\rm an}) \to  0 $ as $ W_{\rm an} \to 0 $ in $\Chi^{r+1}_\sigma$, and there exists a constant $ C_4(M,W_{\rm an}) >0 $, which depends on $M$ and $\lVert W_{\rm an} \rVert_{\Chi^{r+1}_\sigma}$ only, such that for all
$ \varepsilon \in (0,1) $ and $R \in \Chi^{r+1}_\sigma$ with $ \| R \|_{ \Chi^{r+1/2}_\sigma}  \leq M $  we have the estimate
\begin{align*}
&\left|(    \vartheta h_{c}( W_{\rm an}, R) ,R_c)_{G^r_{\sigma}} \right| + \left|( |k|_{op}^{1/2} h_{s}( W_{\rm an}, R) , |k|_{op}^{1/2} R_s)_{G^r_{\sigma}}\right| \\
&\qquad \leq
C_3(W_{\rm an})\left(\||k|_{op}^{1/2} R\|^2_{G^{r+1/2}_{\sigma}} + \varepsilon  \|R\|^2_{G^{r+1}_{\sigma}}\right) + C_4(M,W_{\rm an}) \varepsilon^{\kappa} \|R\|^2_{G^{r+1}_{\sigma}} ,
\end{align*}
\end{lemma}
\noindent{\bf Proof.}
Recall from the proof of Lemma~\ref{webex44} that $g_s$ and $g_c$ consist of bounded linear transformations acting on entire functions of $r, \mathcal{M} v, B, \partial_x B$ and $\partial_x r$. Hence, the same holds for $ h_s $ and $ h_c $ upon replacing $ W $ by $ W_{\rm an} + \varepsilon^{\kappa} R $.
Therefore, we have a representation
$$
h_j( W_{\rm an}, R) = h_{j,lin}( W_{\rm an}, R) + h_{j,non}( W_{\rm an}, R), \qquad j = s,c,
$$
with $ h_{j,lin} $ linear in $ R $ and $ h_{j,non} $ nonlinear in $R$.

As in the proof of Lemma~\ref{webex44}, we arrive at the estimate
$$
\| h_{j,non}( W_{\rm an}, R)  \|_{G^r_{\sigma}}  \leq
 C_4(M,W_{\rm an}) \varepsilon^{\kappa} \|R\|_{G^{r+1}_{\sigma}}, \qquad j = s,c,
$$
on the nonlinear terms in $R$. From the last estimate and
\begin{align*}
|(    \vartheta h_{c,non}( W_{\rm an}, R) ,R_c)_{G^r_{\sigma}}| &\leq \|h_{c}( W_{\rm an}, R)  \|_{G^{r}_{\sigma}} \| R_c \|_{G^{r}_{\sigma}},\\
|(|k|_{op}^{1/2} h_{s,non}( W_{\rm an}, R) ,|k|_{op}^{1/2} R_s)_{G^r_{\sigma}}| &\leq \|h_s( W_{\rm an}, R)  \|_{G^{r}_{\sigma}} \|R_s \|_{G^{r+1}_{\sigma}},
\end{align*}
we immediately obtain the required estimate for the $h_{j,non} $-contribution.

By the considerations in the proof of Lemma~\ref{webex44}, the  parts
$ (  \vartheta h_{c,lin}( W_{\rm an}, R) ,R_c)_{G^r_{\sigma}} $ and $ (|k|^{1/2}_{op} h_{s,lin}( W_{\rm an},  R) ,|k|^{1/2}_{op}R_s)_{G^r_{\sigma}} $ are finite sums of terms of the form
$$
\int_\R\int_\R \overline{\widehat{R}}_{j_1}(k)  s_0(k) s_1(k-l) \widehat{H(W_{\rm an})}(k-l)  s_2(l)  \widehat{R}_{j_2}(l) dl (1+|k|^{2})^r e^{2 \sigma |k|} dk = (*) ,
$$
with $R_{j_1}$ a component of the vector $R_s$ or $R_c$, and $R_{j_2}$ a component of the vector $R$, where $H(W_{\rm an})$ is a bounded linear transformation applied to an entire function of $ W_{\rm an} $,
and where $s_0,s_i \colon \R \to \R$ are functions satisfying $|s_0(k)| \leq C|k|$ and $|s_i(k)| \leq C (1+|k|) $ for $k \in \R$ and $i = 1,2$.
We proceed as in the prototypical example above. First, we observe that
\begin{align*}|s_0(k)&s_1(k-l)s_2(l)| \leq
C|k|\left(1 + |l|\right)(1 + |k-l|)\\
&\leq C\sqrt{|k|}\left(\sqrt{|k-l|} + \sqrt{|l|}\right) + C|k|\left(|k-l| + |l| + |l||k-l|\right), \end{align*}
for $k,l \in \R$. Thus, we estimate $(*)$ by
\begin{align*}
&\int_\R\int_\R |s_0(k)s_1(k-l)s_2(l)| |{\widehat{R}}_{j_1}(k)|  | |\widehat{H(W_{\rm an})}(k-l)| |\widehat{R}_{j_2}(l) | dl (1+|k|^{2})^r e^{2 \sigma |k|} dk
\\
&\qquad \leq C_3(W_{\rm an})\left(\||k|_{op}^{1/2} R\|^2_{G^{r+1/2}_{\sigma}} +  \varepsilon  \|R\|^2_{G^{r+1}_{\sigma}}\right)
\end{align*}
where we used that $  \| |k|^i  \widehat{H(W_{\rm an})}(k) \|_{L^1(dk)} = \mathcal{O}(\varepsilon^{i}) $ for $i = 0,1/2,1$ due to the long-wave character of $ W_{\rm an} $ in physical space. So the required estimate for the $h_{j,lin} $-contribution follows, too.
\qed

\subsection{The final estimates}

We define the energies
$$ E_{j,i}(t) = \|R_j(t)\|^2_{G^i_{\sigma(t)}},$$
for $j = s,c, \ i = r,m$. Moreover, we introduce for $j = s,c, \ i = r,m$
$$ |k|_{op}^{1/2}E_{j,i}(t) = \||k|_{op}^{1/2} R_j(t)\|^2_{G^i_{\sigma(t)}},\qquad |k|_{op}E_{j,i}(t) = \||k|_{op} R_j(t)\|^2_{G^i_{\sigma(t)}}.$$
In the upcoming part in our notation we suppress the $t$-dependency of these quantities.

Since $ R_c $ has compact support in Fourier space by~\eqref{specdecomp} and~\eqref{Rdecomp}, we can estimate its $G^{m+1/2}_{\sigma(t)} $-norm by its $G^r_{\sigma(t)} $-norm, where we recall $r = m-1/2 > 1$. Thus, there exists a $t$-independent constant $ C_{cp} > 0 $ such that
\begin{align}E_{c,m+1/2} \leq C_{cp}E_{c,r}, \qquad |k|_{op}^{1/2}E_{c,r+1/2} \leq C_{cp}|k|_{op}^{1/2}E_{c,r}. \label{loccenter}\end{align}
Finally, we set
$$ E_m = E_{s,m}+ E_{c,m}, \qquad \E_r = |k|_{op}^{1/2} E_{s,r} + E_{c,r}.$$

Taking $n \in \mathbb N$ with $n \geq \kappa + 1$ and $M > 0$, the previous estimates~\eqref{estF1},~\eqref{estF2}, \eqref{estF3} and~\eqref{estF4}, and Lemma~\ref{Lem3.1} from Section~\ref{sec34new} condense in the three inequalities
\begin{align}
\begin{split}
\frac12 \frac{d}{dt} E_{s,m} & \leq -  \eta |k|_{op}^{1/2} E_{s,m}  - d_{\Lambda,0} E_{s,m} +
d_{\Lambda,1} |k|_{op}^{1/2}E_{s,m} + C_{res} \varepsilon (E_{s,m} +1)
\\&\quad +C_1(W_{\rm an}) E_{m+1/2}
+ C_2(M,W_{\rm an}) \varepsilon^{\kappa} E_{m+1/2}, \label{3ineq1}
\end{split}\\
\begin{split}
\frac12 \frac{d}{dt} \left[|k|_{op}^{1/2} E_{s,r}\right] & \leq -\eta |k|_{op} E_{s,r} - d_{\Lambda,0} |k|^{1/2}_{op} E_{s,r} + d_{\Lambda,1} |k|_{op}E_{s,r} + C_{res} \varepsilon (|k|_{op}^{1/2}E_{s,r} +1)
\\&\quad + C_3(W_{\rm an})\left(|k|_{op}^{1/2}E_{r+1/2} + \varepsilon E_{r+1}\right)
+ C_4(M,W_{\rm an}) \varepsilon^{\kappa} E_{r+1},\label{3ineq2}
\end{split}\\
\begin{split}
\frac12 \frac{d}{dt} E_{c,r}  & \leq   - \eta |k|_{op}^{1/2} E_{c,r}  +
d_{\Lambda,1} |k|_{op}^{1/2}E_{c,r} +  C_{res} \varepsilon (E_{c,r} +1)
\\&\quad + C_3(W_{\rm an})\left(|k|_{op}^{1/2}E_{r+1/2} + \varepsilon E_{r+1}\right)
+ C_4(M,W_{\rm an}) \varepsilon^{\kappa} E_{r+1},\label{3ineq3}
\end{split}
\end{align}

We use $E_{m+1/2} \leq E_{s,m} + |k|_{op}^{1/2} E_{s,m} + E_{c,m+1/2}$ and~\eqref{loccenter} to rewrite~\eqref{3ineq1} as
\begin{align*}
\frac12 \frac{d}{dt} E_{s,m} & \leq - \eta |k|_{op}^{1/2} E_{s,m}  - d_{\Lambda,0} E_{s,m} +
d_{\Lambda,1} |k|_{op}^{1/2}E_{s,m} + C_{res} \varepsilon
\\ &\qquad
+ \left(C_1(W_{\rm an}) + C_2(M,W_{\rm an}) \varepsilon^{\kappa} + C_{res} \varepsilon\right)\left(E_{s,m} + |k|_{op}^{1/2} E_{s,m} + C_{cp} E_{c,r}\right).
\end{align*}
Lemma~\ref{webex44} implies that there exists an $M$- and $t$-independent bound $B_1 > 0$ such that, if we have $\|W_{\mathrm{an}}\|_{G^{m+1/2}_{\tilde{\sigma}_1}} < B_1$ and $\varepsilon \in (0,d_{\Lambda,0}/4(1+C_{res}))$, then it holds
\begin{align} \label{assump15.2}
C_1(W_{\rm an}) + \varepsilon + C_{res}\varepsilon < \frac{d_{\Lambda,0}}{2}.
\end{align}
We note that, by Corollary~\ref{Cor2.res}, $\smash{\|W_{\mathrm{an}}\|_{G^{m+1/2}_{\tilde{\sigma}_1}}} < B_1$ can be achieved by  taking $C_{wh} > 0$ sufficiently small in~\eqref{whismal}.
Then, there exists a bound $\epsilon_1 = \smash{\epsilon_1(\kappa,M,\|W_{\rm an}\|_{G_{\tilde{\sigma}_1}^{m+1/2}}) > 0}$, which depends on $\kappa > 1$, $M$ and $\smash{\|W_{\rm an}\|_{G_{\tilde{\sigma}_1}^{m+1/2}}}$ only, such that for $\epsilon \in (0,\epsilon_1)$ we have
\begin{align} \label{assump15.1}
C_2(M,W_{\rm an})\varepsilon^{\kappa - 1} < 1.
\end{align}
where we use $\kappa > 1$. Thus, assuming~\eqref{assump15.2}-\eqref{assump15.1} are satisfied and taking
\begin{align*}
\eta > \frac{d_{\Lambda,0}}{2} + d_{\Lambda,1},
\end{align*}
yields
\begin{align}
\label{Esmineq}
\begin{split}
\frac12 \frac{d}{dt} E_{s,m} &\leq -\frac{d_{\Lambda,0}}{2} E_{s,m} + C_{res} \varepsilon
 + \frac{d_{\Lambda,0}}{2} C_{cp} E_{r,c} \\
 &\leq -\frac{d_{\Lambda,0}}{2} E_{s,m} + C_{res} \varepsilon
 + \frac{d_{\Lambda,0}}{2} C_{cp} \E_{r}.
 \end{split}
\end{align}

Subsequently, we combine \eqref{3ineq2}-\eqref{3ineq3} and use the identities $|k|_{op}^{1/2} E_{r+1/2} \leq |k|_{op}^{1/2} E_{s,r} + |k|_{op} E_{s,r} + |k|^{1/2}_{op} E_{c,r+1/2}$ and $\smash{E_{r+1} \leq E_r + |k|_{op}^{1/2} E_{r+1/2}}$, in combination with~\eqref{loccenter} to arrive at
\begin{align*}
\frac12 \frac{d}{dt} \E_{r}  & \leq  -\eta |k|_{op}^{1/2} E_{c,r}  +
d_{\Lambda,1} |k|_{op}^{1/2}E_{c,r} -  \eta |k|_{op} E_{s,r} - d_{\Lambda,0} |k|^{1/2}_{op} E_{s,r} +
d_{\Lambda,1} |k|_{op}E_{s,r} \\
& + 2\left(C_3(W_{\rm an})(1 + \varepsilon) + C_4(M,W_{\mathrm{an}})\varepsilon^\kappa \right)\left(C_{cp}|k|_{op}^{1/2}E_{c,r} + |k|_{op}^{1/2}E_{s,r} + |k|_{op}E_{s,r}\right)
 \\& + 2\left(C_3(W_{\rm an}) \varepsilon + C_4(M,W_{\mathrm{an}})\varepsilon^\kappa \right) E_r + 2C_{res} \varepsilon + C_{res}\varepsilon\E_{r}.
\end{align*}
By Lemma~\ref{Lem3.1} there exists an $M$- and $t$-independent bound $B_2 > 0$ such that, if we have $\|W_{\mathrm{an}}\|_{G^{r+1}_{\tilde{\sigma}_1}} < B_2$ and $\varepsilon \in (0,\frac{d_{\Lambda,0}}{4})$, then it holds
\begin{align} \label{assump17.2}
C_3(W_{\mathrm{an}})(1+\varepsilon) + \varepsilon < \frac{d_{\Lambda,0}}{2}.
\end{align}
As before, this can be achieved by choosing $C_{wh}$ sufficiently small.
Moreover, Lemma~\ref{Lem3.1} implies that there exists a bound $\smash{\epsilon_2 = \epsilon_2(\kappa,M,\|W_{\mathrm{an}}\|_{G^{r+1}_{\tilde{\sigma}_1}}) > 0}$, which depends on $\kappa > 1$, $M$ and $\smash{\|W_{\rm an}\|_{G_{\tilde{\sigma}_1}^{r+1}}}$ only, such that for $\epsilon \in (0,\epsilon_2)$ we have
\begin{equation} \label{assump17.1}
C_4(M,W_{\mathrm{an}}) \varepsilon^{\kappa-1} < 1.
\end{equation}
where we recall $r = m-1/2$.
Thus, assuming~\eqref{assump17.2}-\eqref{assump17.1} and taking $\eta > 0$ so large that
\begin{equation*}
\eta > d_{\Lambda,1}  + (1 + C_{cp})d_{\Lambda,0},
\end{equation*}
yields
\begin{align}
\label{Erineq}
\begin{split}
\frac12 \frac{d}{dt} \E_{r} &\leq 2\left(C_3(W_{\rm an}) + C_4(M,W_{\mathrm{an}})\varepsilon^\kappa \right) \varepsilon  E_r +  2C_{res} \varepsilon + C_{res}\varepsilon\E_{r} \\
&\leq \left(d_{\Lambda,0} + C_{res}\right)\varepsilon\left(E_{s,m} + \E_r\right) +  2C_{res} \varepsilon.
\end{split}
\end{align}

Our next step is to integrate the system of differential inequalities~\eqref{Esmineq} and~\eqref{Erineq}. We introduce
$$ \mathcal{S}_{s}(\tau) = \sup_{\tilde{t} \in [0,\tau]} E_{s,m}(\tilde{t})
\quad \text{and} \quad
 \widetilde{\mathcal{S}}(\tau) = \sup_{\tilde{t} \in [0,\tau]} \E_{r}(\tilde{t}). $$
Integrating and applying Gr\"onwall's inequality to~\eqref{Esmineq}, we arrive at
$$
 \mathcal{S}_{s}(\tau)  \leq C ( \widetilde{\mathcal{S}}(\tau) + \varepsilon), \qquad \tau \in [0,t],
$$
since $W(0) = W_{\rm{\mathrm{an}}}(0)$.
We insert the last inequality in~\eqref{Erineq} and integrate to find
$$
 \widetilde{\mathcal{S}}(t) \leq 2\int_0^t \left(\left(d_{\Lambda,0} + 2\right)\varepsilon(C+1)\left(\widetilde{\mathcal{S}}(\tau) + \varepsilon\right) +  2C_{res} \varepsilon\right)  d \tau.
$$
Gr\"onwall's inequality finally gives bounds $ M_s, M_c > 0$, which are independent of $\varepsilon \in \min\{\epsilon_1, \epsilon_2, d_{\Lambda,0}/4(1+C_{res}\}$, $M$ and $t$ as long as $t \in (0,T_1/\epsilon]$, such that
\begin{align} \label{finale}
\mathcal{S}_{s}(t) \leq M_s \qquad \textrm{and} \qquad  \widetilde{\mathcal{S}}(t) \leq M_c.
\end{align}

We set $M = (M_s+M_c)^{1/2}$ and apply continuous induction to~\eqref{bounded} using~\eqref{finale}. We conclude that the error remains bounded (and thus exists, cf.~Remark~\ref{rem:local}) for all $t \in (0,T_{1,\veps}/\veps]$.
Thus, we establish~\eqref{regV}, because~\eqref{locexappr},~\eqref{regR} and the embeddings $G_{\sigma(t)}^m \hookrightarrow H^m(\R)$ and $\smash{G_{\sigma(t)}^{m+1/2} \hookrightarrow H^{m+2}(\R)}$ hold for all $t \in (0,T_1/\veps]$.
Finally, upon recalling the prefactor $\veps^\kappa$ in~\eqref{deferror}, we arrive at the main error estimate~\eqref{mr2} by the fact that~\eqref{bounded} and the embedding $G_{\sigma(t)}^m \hookrightarrow L^\infty(\R)$ hold for all for $t \in (0,T_1/\veps]$.
All in all, this proves our main result, Theorem~\ref{mainresult1}.
\qed

\begin{remark}{\upshape
In the above proof of Theorem~\ref{mainresult1}, we obtained a system of differential inequalities
\begin{align*}
 \frac12 \frac{d}{dt}
 \begin{pmatrix}
  E_{s,m} \\ \E_{r}
 \end{pmatrix}
 &\leq
\mathcal{A}
 \begin{pmatrix}
  E_{s,m} \\ \E_{r}
 \end{pmatrix}
 +
 C_{res} \varepsilon
 \begin{pmatrix}
  1 \\ 2
 \end{pmatrix}, \qquad \mathcal{A} := \begin{pmatrix}
  -\frac{d_{\Lambda,0}}{2} & \frac{d_{\Lambda,0}}{2} C_{cp} \\
  \left(d_{\Lambda,0} + C_{res}\right)\varepsilon & \left(d_{\Lambda,0} + C_{res}\right)\varepsilon
 \end{pmatrix},
\end{align*}
where the inequality $\leq$ should be understood componentwise. We emphasize that the conclusion~\eqref{finale} follows by the properties of the matrix $\mathcal A$. Indeed, the componentwise order in $\R^2$ is preserved by the associated matrix exponential $e^{\mathcal At}$. Hence, integrating the system and observing that the eigenvalues of $\mathcal{A}$ are bounded from above by $C \epsilon$, the estimate~\eqref{finale} readily follows.
}\end{remark}

\begin{remark} \label{rem:local} {\upshape
Local existence of the exact solution $W$ to the semilinear parabolic system~\eqref{webex1} in time-independent Gevrey spaces immediately transfers to a time-dependent scale of Gevrey spaces. In particular, there exists a maximal $t_\veps \in (0,\infty]$ and a solution
\begin{align}W \in C^1\big((0,t],G^m_{\sigma(t)}\big) \cap C\big([0,t],G^{m+2}_{\sigma(t)}\big), \qquad t \in (0,t_\veps),\label{locpar}
\end{align}
to~\eqref{webex1} with $W(0) = W_{\rm an}(0)$. If $t_{\veps} < \infty$, then it must hold $$\sup_{t \in [0, t_{\veps})} \|W(t)\|_{G^m_{\sigma(t)}} = \infty.$$ Indeed, if not, then, by the local theory in time-independent Gevrey spaces, there would exist a $\tau > 0$, which is independent of $t$,  such that for any $t \in [0,t_{\veps})$ there is a solution
$$W_t \in C^1\big((t,t + \tau],G^m_{\sigma(t)}\big) \cap C\big((t,t+\tau],G^{m+2}_{\sigma(t)}\big),$$
to~\eqref{webex1} with $W_t(t) = W(t)$. So, in particular, the solution $W$ to~\eqref{webex1} could be extended such that~\eqref{locpar} would hold for $t \in [0,t_\veps + \tau/2)$ contradicting the maximality of $t_\veps$.

Combining the latter with~\eqref{deferror} and~\eqref{locexappr}, it follows that there exists $t_{1,\veps} \in (0,T_1/\veps]$ such that~\eqref{regR} holds for all $t \in (0,t_{1,\veps})$. If $t_{1,\veps} < T_1/\veps$, then we must have
\begin{align*}\sup_{t \in [0,t_{1,\veps})}\|R(\tau)\|_{G^m_{\sigma(t)}} = \infty.\end{align*}
Thus, the error $R(t)$ exists in the relevant time-dependent scale of Gevrey spaces for $t \in (0,T_1/\veps]$ as long as $\|R(t)\|_{G^m_{\sigma(t)}}$ remains bounded.
}\end{remark}

\bibliographystyle{alpha}
\bibliography{GLbib}
\small

\end{document}